\documentclass[a4paper,twoside,english]{amsart}
\usepackage{graphicx}  
\usepackage{amsmath,amssymb,amsthm, amscd}
\usepackage[pagewise]{lineno}
\usepackage{amssymb,fge}
\usepackage{tikz-cd}
\usepackage[T1]{fontenc}
\usepackage[latin9]{inputenc}
\usepackage{babel}
\usepackage{verbatim}
\newcommand{\RomanNumeralCaps}[1]
    {\MakeUppercase{\romannumeral #1}}
\makeatletter
\let\@wraptoccontribs\wraptoccontribs
\makeatother

\usepackage{lipsum}

\newcommand\blfootnote[1]{%
  \begingroup
  \renewcommand\thefootnote{}\footnote{#1}%
  \addtocounter{footnote}{-1}%
  \endgroup
}

\usepackage{color}

\usepackage[left=3.2cm,right=3.5cm,top=3cm,bottom=3cm]{geometry}
\usepackage{indentfirst}           
\usepackage{underscore}
\usepackage{enumitem}
\usepackage{relsize}
\usepackage{varwidth}
\usepackage{mathtools} 
\usepackage{hyperref}
\hypersetup{
    colorlinks=true,
    linkcolor=blue,
    filecolor=magenta,      
    urlcolor=cyan,
}
\usepackage{amssymb}
\usepackage[all]{xy}
\makeatletter
\renewcommand*\env@matrix[1][*\c@MaxMatrixCols c]{%
  \hskip -\arraycolsep
  \let\@ifnextchar\new@ifnextchar
  \array{#1}}
\makeatother

\def\O{\mathcal{O}}
\newcommand{\C}{{\mathfrak{C}}}
\newcommand{\D}{{\mathfrak{D}}}
\def\p{{\mathfrak{p}}}
\newcommand\FPP{\text{FPP}}

\newenvironment{Gal}{{\rm Gal\,}}{}

\def\p{{\mathfrak{p}}}

\def\O{{\mathcal{O}}}

\def\O{{\mathcal O}}

\def\Gal{\mathop{\rm Gal}\nolimits}

\def\O{{\mathcal O}}

\def\CVD{{\hfill\hfil{\lower 2pt\hbox{\vrule\vbox to 7pt
{\hrule width  5pt\varphifill\hrule}\varphirule}}}\par}

\DeclareMathOperator{\SurArb}{SurArb}
\DeclareMathOperator{\BigArb}{BigArb}

\DeclareMathOperator{\Orb}{Orb}

\DeclareMathOperator{\disc}{disc}

\newenvironment{Res}{{\rm Res\,}}{}

\newenvironment{Aut}{{\rm Aut}}{}

\usepackage{thmtools, thm-restate}

\newtheorem{theorem}{Theorem}[section]
\newtheorem{lemma}[theorem]{Lemma}
\newtheorem{proposition}[theorem]{Proposition}
\newtheorem{proposition-definition}[theorem]{Proposition-Definition}
\newtheorem{corollary}[theorem]{Corollary}
\newtheorem{conjecture}[theorem]{Conjecture}

\theoremstyle{definition}
\newtheorem{definition}[theorem]{Definition}
\theoremstyle{remark}
\newtheorem{remark}{Remark}

\theoremstyle{theorem}
\newtheorem{thm}{Theorem}
\newtheorem{lem}[thm]{Lemma}
\newtheorem{prop}[thm]{Proposition}

\theoremstyle{remark}

\newcommand*{\Scale}[2][4]{\scalebox{#1}{$#2$}}%

\newcommand{\Mod}[1]{\ (\textup{mod}\ #1)}     

\title[Random Arboreal Representations]{Galois groups and prime divisors in random quadratic sequences}   
\author[J. R. Doyle, V. O. Healey, W. Hindes, and R. Jones]{John R. Doyle, Vivian Olsiewski Healey, Wade Hindes, and Rafe Jones}
\begin{document} 
\maketitle
\blfootnote{2020 \emph{Mathematics Subject Classification}: Primary: 11R32, 37P15, 11F80. Secondary: 11D99.}

\begin{abstract} Given a set $S=\{x^2+c_1,\dots,x^2+c_s\}$ defined over a field and an infinite sequence $\gamma$ of elements of $S$, one can associate an arboreal representation to $\gamma$, generalizing the case of iterating a single polynomial. We study the probability that a random sequence $\gamma$ produces a ``large-image'' representation, meaning that infinitely many subquotients in the natural filtration are maximal.  We prove that this probability is positive for most sets $S$ defined over $\mathbb{Z}[t]$, and we conjecture a similar positive-probability result for suitable sets over $\mathbb{Q}$. As an application of large-image representations, we prove a density-zero result for the set of prime divisors of some associated quadratic sequences. We also consider the stronger condition of the representation being finite-index, and we classify all $S$ possessing a particular kind of obstruction that generalizes the post-critically finite case in single-polynomial iteration.
\end{abstract}

\section{Introduction} 
Let $K$ be a field, let $S$ be a fixed set of polynomials over $K$, and let $\gamma=(\theta_1,\theta_2,\dots)$ be an infinite sequence of elements $\theta_i\in S$. Then we are interested in the tower of field extensions $K_n(\gamma):=K(\theta_1\circ\theta_2\circ\dots\circ\theta_n)$, where $K(f)$ denotes the splitting field of $f\in K[x]$ in a fixed algebraic closure $\overline{K}$. In particular, and under some mild separability assumptions, the associated Galois groups $G_{\gamma,n,K}:=\Gal(K_n(\gamma)/K)$ act naturally on the corresponding preimage trees,       
\[T_{\gamma,n}:=\big\{\alpha\in\overline{K}\,:\,\theta_1\circ\theta_2\circ\dots\circ\theta_m(\alpha)=0\; \text{for some $1\leq m\leq n$}\big\}.\]
Here the edge relation is given by the rule: if $\theta_1\circ\dots\circ \theta_m(\alpha)=0$, then there is an edge between $\alpha$ and $\theta_m(\alpha)$. In particular, since Galois groups over $K$ commute with evaluation of polynomials over $K$, the inverse limit of groups
\[G_{\gamma,K}:=\lim_{\longleftarrow} G_{\gamma,n,K}\] 
(whose connecting maps are given by restriction) acts continuously on the complete preimage tree $T_\gamma=\bigcup_{n\geq1} T_{\gamma,n}$. Hence, we obtain an embedding,    
\[G_{\gamma,K}\leq \Aut(T_\gamma),\]
called the \emph{arboreal representation} of $\gamma$ (rooted at $0$); see \cite[\S2]{Ferraguti} and Section \ref{sec:arboreals} below for more details. 

The case of constant sequences (corresponding to iterating a single function) for polynomials of small degree has obtained much interest in recent years; see, for example, \cite{cubic:abc,Ferraguti:quad,Riccati, Jones}. In these cases, it is believed that $G_{\gamma,K}$ is a finite index subgroup of $\Aut(T_\gamma)$ (or a smaller overgroup \cite{arboreal:general}), outside of a moderate list of obstructions. However, for general sets $S$ containing at least two polynomials, there are infinitely many possible sequences each of which furnish their own representations. Moreover in practice, many (or even most) of these sequences avoid the corresponding obstructions to finite index. To test this heuristic, we consider sets of the form $S=\{x^2+c_1,\dots, x^2+c_s\}$, so that the ramification in the fields $K_n(\gamma)$ is controlled by a single semigroup orbit (of the common critical point $0$). Moreover, many of the techniques used for constant sequences \cite{Jones-io, Jones, Jones-Survey} admit suitable generalizations in this case; see \cite[\S6]{Me:LeftRightTotal} and Section \ref{sec:arboreals} below. Finally, to make precise what we mean by ``many sequences'', we fix a probability measure $\nu$ on $S$ and let $\bar{\nu}=\nu^{\mathbb{N}}$ be the product measure on $\Phi_S=S^{\mathbb{N}}$, the set of all infinite sequences of elements of $S$. In particular, a property $P$ holds for ``many'' sequences in $S$ if it holds with positive probability: $\bar{\nu}\big(\{\gamma\in\Phi_S\,:\, \text{$\gamma$ has property $P$}\}\big)>0$.  

A first task with this more general setup is to identify what properties of $S$ are obstructions to producing finite index representations with positive probability. Certainly, as in the case of iterating a single function, if $K_\infty(\gamma)=\bigcup_n K_n(\gamma)$ is a finitely ramified extension of $K$, then $G_{\gamma,K}$ is an infinite index subgroup of $\Aut(T_\gamma)$; see \cite[Theorem 3.1]{Jones-Survey}. In particular, if the full semigroup orbit of $0$ is finite, then the discriminant formula in \cite[Proposition 6.2]{Me:LeftRightTotal} implies infinite index for all sequences. Likewise, with a little background in the theory of probability, one can see that a similar problem will arise with a weaker property: when the semigroup orbit of $0$ \emph{contains} a point whose orbit is finite (even though the full orbit of $0$ may be infinite). However perhaps surprisingly, one can write down a complete list of such sets over the rational numbers, using previous work in \cite{Me:FiniteOrbit} on finite orbit points. In particular, we have the following complete classification of this obstruction to finite index; in what follows, $\Orb_S(Q)$ denotes the full semigroup orbit of the point $Q\in K$ generated by the maps in $S$ under composition. Furthermore, we say $\nu$ is \emph{strictly positive} if $\nu(\phi)>0$ for all $\phi\in S$.  \vspace{.1cm}       
\begin{theorem}\label{thm:obstruction} Let $S=\{x^2+c_1,\dots, x^2+c_s\}$ be a set of quadratic polynomials with rational coefficients, let $\nu$ be a strictly positive probability measure on $S$, and let $\bar{\nu}:=\nu^{\mathbb{N}}$ be the associated product measure on $\Phi_S:=S^{\mathbb{N}}$. Then the following statements hold: \vspace{.1cm}  
\begin{enumerate} 
\item[\textup{(1)}] If $\Orb_S(0)$ contains a finite orbit point, then $S$ is one of the following exceptional sets:\vspace{.1cm} 
\[\qquad\qquad S=\big\{x^2\big\},\; \big\{x^2-1\big\},\; \big\{x^2-2\big\},\; \big\{x^2,\,x^2-1\big\},\;\big\{x^2-2,\, x^2-3\big\}\;\textup{or}\; \big\{x^2-2,\, x^2-6\big\}. \vspace{.15cm}\] 
\item[\textup{(2)}] Let $\Phi_S^{\textup{sep}}\subseteq\Phi_S$ be the set of sequences $\gamma$ such that $\gamma_n$ is separable for all $n$. If $S$ is one of the sets in \textup{(1)}, then \[\bar{\nu}\Big(\big\{\gamma\in\Phi_S^{\textup{sep}}\,:\, [\Aut(T_{\gamma}):G_{\gamma,\mathbb{Q}}]<\infty \big\}\Big)=0. \vspace{.05cm} \] 
\end{enumerate} 
In particular, if $\Orb_S(0)$ contains a finite orbit point for $S$, then a random sequence $\gamma$ furnishes a finite index arboreal representation with probability zero.  
\end{theorem}
Although the classification above is a step in the right direction, it is unclear at the moment what (if any) other obstructions to producing finite index representations with positive probability remain; we plan to return to this problem at a later date. On the other hand, there is a weaker and more approachable property than finite index, and in certain circumstances, this property is enough to prove density-zero results for prime divisors in orbits; see Theorem \ref{thm:density} below. Namely, we seek sequences $\gamma=(\theta_n)_{n\geq1}$ for which all of the finite level polynomials $\gamma_n=\theta_1\circ\dots\circ\theta_n$ are irreducible over $K$ and for which the subextensions $K_n(\gamma)/K_{n-1}(\gamma)$ are as large as possible for infinitely many $n$. With this in mind, given a sequence $\gamma$ of quadratic polynomials we say that an extension $K_n(\gamma)/K_{n-1}(\gamma)$ is \emph{maximal} if $[K_n(\gamma):K_{n-1}(\gamma)]=2^{2^{n-1}}$; see Remark \ref{rem:maximal} for justification of this language. Moreover, by analogy with the case of constant sequences \cite[\S4]{Jones}, we say that a sequence $\gamma=(\theta_n)_{n\geq1}$ is \emph{stable} over $K$ if $\gamma_n=\theta_1\circ\dots\circ\theta_n$ is irreducible over $K$ for all $n\geq1$. Finally combining these two notions, we say that a sequence $\gamma$ furnishes a \emph{big arboreal representation} over $K$ if $\gamma$ is stable over $K$ and $K_n(\gamma)/K_{n-1}(\gamma)$ is maximal infinitely often. Moreover, we let   
\vspace{.1cm}
\begin{equation}\label{eq:big}
\BigArb(S,K):=\Big\{\gamma\in\Phi_S\,:\,\text{$\gamma$ is stable over $K$ and $K_n(\gamma)/K_{n-1}(\gamma)$ is maximal i.o.}\Big\}
\end{equation}
be the set of infinite sequences in $\Phi_S$ that furnish big arboreal representations over $K$. In particular, based on analogy with the case of iterating a single function \cite{Ferraguti:quad,Jones-Survey}, heuristics on the growth rates of heights in sequential orbits \cite{Me:stochastic,Me:LeftRightTotal,Kawaguchi}, and unconditional results achieved over $\mathbb{Z}[t]$ below, we conjecture that a positive proportion of sequences furnish big arboreal representations over $\mathbb{Q}$, as long as the generating set $S$ has at least $3$ elements, two of which are irreducible:   

\begin{conjecture}\label{conj2} Let $S=\{x^2+c_1,\dots, x^2+c_s\}$ be a set of quadratic polynomials over $\mathbb{Q}$ and let $\nu$ be any strictly positive probability measure on $S$. Moreover, assume that $S$ contains at least $3$ elements, two of which are irreducible in $\mathbb{Q}[x]$. Then $\bar{\nu}(\BigArb(S,\mathbb{Q}))>0$. 
\end{conjecture} 
\begin{remark} It was recently shown in \cite{Mathworks} that if $S=\{x^2+c_1,\dots, x^2+c_s\}$ for some $c_i\in\mathbb{Z}$, then the $\bar{\nu}$ measure of the set of $\mathbb{Q}$-stable sequences in $S$ is positive. In particular, some progress on Conjecture \ref{conj2} has been made for quadratic polynomials with integral coefficients.  
\end{remark} 
To give some evidence for Conjecture \ref{conj2}, we replace $\mathbb{Q}$ with the polynomial ring $\mathbb{Z}[t]$ and prove a similar statement for ``most sets'' $S$ in this setting, at least if the cardinality of $S$ is large enough. To make this idea precise, we fix some notation. Given a polynomial $f=a_dt^d+\dots +a_1t+a_0\in\mathbb{Z}[t]$, we define $|f|=\max_{0\leq i\leq d}\{|a_i|\}$ to be the maximum absolute value of $f$'s coefficients, and set \vspace{.05cm} 
\[P_d(B)=\{f\in\mathbb{Z}[t]\,:\;\deg(f)\leq d\;\text{and}\; |f|\leq B\}. \vspace{.05cm} \]  
Likewise for any fixed $s\geq1$, define \vspace{.1cm}  
\begin{equation}\label{eq:sets}
\mathcal{S}(d,s,B):=\{\{c_1,\dots,c_s\}\,:\, c_i\in P_d(B)\} \vspace{.1cm}
\end{equation} 
to be the collection of sets with $s$-elements chosen from $P_d(B)$. In particular, given an element $\{c_1,\dots,c_s\}\in \mathcal{S}(d,s,B)$ we associate a set of quadratic polynomials with coefficients in $\mathbb{Z}[t]$, \vspace{.05cm}  
\[ S=S\big(\{c_1,\dots,c_s\}\big)=\{x^2+c_1,\dots, x^2+c_s\}, \vspace{.05cm} \] 
and study the sequences in $S$ furnishing big representations over $K=\mathbb{Q}(t)$. In particular, we prove that for any fixed $d$ and large $s$ (depending on $d$), most sets in $\mathcal{S}(d,s,B)$ furnish big arboreal representations with positive probability as $B\rightarrow\infty$. That is, an analog of Conjecture \ref{conj2} holds for almost all sets $S=\{x^2+c_1,\dots, x^2+c_s\}$ with $\deg(c_i)\leq d$ (asymptotically full density in $\mathcal{S}(d,s,B)$ as $B\rightarrow\infty$) in the large $s$ limit: \vspace{.1cm} 
\begin{restatable}{theorem}{thmnonmonic}
\label{thm:nonmonic}
Let $d>0$ and $s\geq2$, let $\BigArb(S)$ and $\mathcal{S}(d,s,B)$ be as in \eqref{eq:big} and \eqref{eq:sets} above for $K=\mathbb{Q}(t)$, and let \vspace{.1cm}   
\[r_d:=
\begin{cases} 
      \big(\frac{1}{2}\big)^{\frac{d}{2}+1} & \text{$d$ is even,} \\[3pt]
       \big(\frac{1}{2}\big)^{\frac{d+1}{2}} & \text{$d$ is odd.}
   \end{cases}
\]
Then the following statements hold: \vspace{.2cm}
\begin{enumerate} 
\item[\textup{(1)}] If $d$ is even, then \vspace{.15cm}
\[\liminf_{B\rightarrow\infty}\frac{\#\Big\{S\in\mathcal{S}(d,s,B)\,:\,\bar{\nu}_S\big(\BigArb(S,K)\big)>0\Big\}}{\#\mathcal{S}(d,s,B)}\geq 1-(1-r_d)^s. \vspace{.25cm}\]  
\item[\textup{(2)}] If $d$ is odd, then \vspace{.2cm} 
\[\;\;\displaystyle{\liminf_{B\rightarrow\infty}}\frac{\#\Big\{S\in \mathcal{S}(d,s,B)\,:\,\bar{\nu}_S\big(\BigArb(S,K)\big)>0\Big\}}{\#\mathcal{S}(d,s,B)}\geq 1-(1-r_{d})^s-\Big(\frac{1}{2}\Big)^s+\Big(1-r_d-\frac{1}{2}\Big)^s. 
\vspace{.4cm}\]  
\end{enumerate} 
In particular, when $d\geq1$ and $s\geq2$ are fixed, the number of sets $S=\{x^2+c_1,\dots,x^2+c_s\}$ with $\deg(c_i)\leq d$ and which furnish with positive probability big arboreal representations approaches full density (in the set of all possible $S=\{x^2+c_1,\dots,x^2+c_s\}$ with $\deg(c_i)\leq d$) as $s$ grows.\vspace{.1cm} 
\end{restatable} 
\begin{remark} Likewise, given a set $S$ we can study the sequences in $S$ furnishing surjective arboreal representation over $K=\mathbb{Q}(t)$. In particular, if we assume (for technical reasons only) that the defining polynomials in $S$ are \emph{monic and of even degree}, then we prove surjectivity with positive probability for most sets over $\mathbb{Z}[t]$; see Theorem \ref{thm:monic} in Section \ref{sec:Z[t]} below.   
\end{remark}   
Our results over $\mathbb{Z}[t]$ are based upon the following convenient maximality test for sets. Interestingly, the strategy of the proof of the statement below builds upon an earlier argument in \cite[Theorem 1.3]{Me:AvgZig}, which proves that the Galois groups of the iterates of the specific polynomials $\phi(t)=x^d+t$ for $d\geq 2$ are the full wreath product of cyclic groups of order $d$. However, in this case one must first adjoin the $d$-th roots of unity to the the base field $\mathbb{Q}(t)$.      
\begin{restatable}{theorem}{thmtool}
\label{thm:tool}
Let $S=\{x^2+c_1,\dots, x^2+c_s\}$ for some polynomials $c_i\in \mathbb{Z}[t]$ and suppose that the following conditions hold:\vspace{.15cm}
\begin{enumerate} 
\item[\textup{(1)}] Some $c_j$ satisfies $\frac{d}{dt}(\overline{c_j})=1$ in $\mathbb{F}_2[t]$. \vspace{.2cm} 
\item[\textup{(2)}] Some $c_k$ with odd leading term satisfies $\deg(c_k)=\max\{\deg(c_1), \dots,\deg(c_s)\}$. \vspace{.15cm} 
\end{enumerate} 
Then $\gamma=(\theta_n)_{n\geq1}$ is $\mathbb{Q}(t)$-stable and $K_n(\gamma)/K_{n-1}(\gamma)$ is maximal if $\theta_1(0)=c_j$ and $\theta_n(0)=c_k$.
\end{restatable}   
\begin{remark} This result can be applied to many singleton sets as well. For example, Theorem \ref{thm:tool} implies that the arboreal representations of $\phi(x)=x^2+t$ and $\phi(x)=x^2+(t^2-3t)$ are surjective over $\mathbb{Q}(t)$. On the other hand, it also implies surjectivity with positive probability for sequences generated by many non-singelton sets, like $S=\big\{x^2+(t^4+5t), x^2-(7t^4+3)\big\}$.     
\end{remark} 
\begin{remark} The idea behind the proof of Theorem \ref{thm:tool} is the following: conditions (1) and (2) together imply that $\gamma_n(0)$ is square-free in $\mathbb{Q}(t)$ for all $n\geq1$. Moreover, the degree condition in (2) coupled with the fact that $\gamma_n(0)$ is square-free implies that $\gamma_n(0)$ has a \emph{primitive prime divisor appearing to odd valuation}; compare to \cite[Theorem 3.3]{Jones} or \cite{abc:primdiv}. The claim then follows from a generalization of Stoll's original maximality criterion \cite[Lemma 1.6]{Stoll}; see Theorem \ref{thm:GaloisMax} below.    
\end{remark}
Finally, as motivation for Conjecture \ref{conj2}, Theorem \ref{thm:nonmonic}, and the study of big arboreal representations in general, we prove a density-zero result for the set of  prime divisors of some associated quadratic sequences. To state this result, let $K$ be a number field and let $\O_K$ be the ring of integers in $K$. Then for $\gamma=(\theta_n)_{n\geq1}$ with $\theta_i \in K[x]$, we consider sequences in $K$ of the form $(\gamma_n(a_0))_{n \geq 0}$, where $a_0 \in K$, $\gamma_0(x) = x$, and $\gamma_n(x) = (\theta_1 \circ \cdots \circ \theta_n)(x)$ for $n \geq 1$. In particular, we are interested in the set of prime ideal divisors of these sequences, namely \vspace{.1cm} 
\[P(\gamma, a_0) := \{\p \subset \O_K : \text{$\p$ is prime and $\p \mid \gamma_n(a_0)$ for at least one $n \geq 0$ with $\gamma_n(a_0) \neq 0$}\}. \vspace{.1cm}\] 
More specifically, we would like to measure the size of  $P(\gamma, a_0)$ by computing its density; recall that the natural density of a set $T$ of primes in $\O_K$ is \vspace{.05cm} 
\begin{equation*} 
D(T) = \lim_{x \to \infty} \frac{\#\{\p \in T : N(\p) \leq x\}}{\#\{\p : N(\p) \leq x\}}, \vspace{.05cm} 
\end{equation*}
provided that this limit exists. Here $N(\p)$ denotes the norm of $\p$. In particular, we prove that $P(\gamma, a_0)$ has density zero whenever $S = \{x^2 + c_1, \ldots, x^2 + c_s\}$ and $\gamma$ furnishes a big arboreal representation over $K$; compare to \cite[Theorem 1.3]{Jones-io}.

 \vspace{.1cm}
\begin{restatable}{theorem}{thmdensity}
\label{thm:density}
Let $K$ be a number field and let $S = \{x^2 + c_1, \ldots, x^2 + c_s\}$ with $c_i \in K$. Suppose that $\gamma \in \textup{BigArb}(S,K)$. Then $D(P(\gamma, a_0)) = 0$ for any $a_0 \in K$. \vspace{.1cm} 
\end{restatable}  
An outline of our paper is as follows. In Section \ref{sec:arboreals}, we record some generalizations of the standard stability and maximality tools for iterating a single function. In Section \ref{sec:finiteorbit}, we classify those exceptional sets of quadratic polynomials over the rationals for which $\Orb_S(0)$ contains a finite orbit point; see part (1) of Theorem \ref{thm:obstruction}. In Section \ref{sec:arboverZ}, we prove that these exceptional sets produce finite index arboreal representations with probability zero; see part (2) of Theorem \ref{thm:obstruction}. In Section \ref{sec:Z[t]}, we study arboreal representations over $\mathbb{Z}[t]$ and prove the aforementioned results in this setting. Finally in Section \ref{sec:density}, we prove Theorem \ref{thm:density} on the density of primes divisors in quadratic sequences attached to big arboreal representations.    
\\[5pt] 
\noindent\textbf{Acknowledgements:} We thank the anonymous referee for their many helpful comments.    
   
\section{Stability and Maximality Tools}\label{sec:arboreals}
In this section, we record some useful tools for analyzing quadratic arboreal representations. The statements below (and their justifications) are similar to those for iterating a single function; see \cite[\S6]{Me:LeftRightTotal} for proofs of these facts. In particular, the first result that we need is a convenient irreducibility test for iterates; see \cite[Proposition 6.3]{Me:LeftRightTotal}. In what follows, given a sequence $\gamma=(\theta_n)_{n\geq1}$ and a positive integer $n$, we let $\gamma_n=\theta_1\circ\dots\circ\theta_n$. \vspace{.1cm}      
\begin{proposition}\label{prop:stability} Let $K$ be a field of characteristic not $2$, let $S=\{x^2+c_1,\dots, x^2+c_s\}$ for some $c_i\in K$, and suppose that $\gamma\in \Phi_S$ satisfies the following properties: \vspace{.05cm} 
\begin{enumerate} 
\item[\textup{(1)}] $-\gamma_1(0)$ is not a square in $K$, \vspace{.1cm} 
\item[\textup{(2)}] $\gamma_n(0)$ is not a square in $K$ for all $n\geq2$. \vspace{.1cm}
\end{enumerate}
Then $\gamma_n$ is irreducible in $K[x]$ for all $n\geq1$.   \vspace{.05cm} 
\end{proposition}
The next tool that we need is a way to determine when the subextensions $K_{n}(\gamma)/K_{n-1}(\gamma)$ are maximal. The following proposition is a generalization of Stoll's original maximality criterion \cite[Lemma 1.6]{Stoll}; see \cite[Proposition 6.7]{Me:LeftRightTotal} for a proof. \vspace{.1cm}   
\begin{proposition}\label{prop:maximality} Let $K$ be a field of characteristic not $2$, let $S=\{x^2+c_1,\dots, x^2+c_s\}$ for some $c_i\in K$, and let $\gamma\in\Phi_S$. If $\gamma_{n-1}$ is irreducible over $K$ for some $n\geq1$, then the following statements are equivalent: \vspace{.15cm}
\begin{enumerate}
\item[\textup{(1)}] $[K_n(\gamma):K_{n-1}(\gamma)]=2^{2^{n-1}}$. \vspace{.15cm}  
\item[\textup{(2)}] $\gamma_n(0)$ is not a square in $K_{n-1}(\gamma)$.  \vspace{.1cm} 
\end{enumerate} 
\end{proposition}
\begin{remark}\label{rem:maximal} Since $K_n(\gamma)/K_{n-1}(\gamma)$ is the compositum of at most $2^{n-1}$ quadratic extensions of $K_{n-1}(\gamma)$, one for each root of $\gamma_{n-1}$, we see that $[K_n(\gamma):K_{n-1}(\gamma)]=2^{2^m}$ for some $0\leq m\leq n-1$. For this reason, when $m=n-1$ we say that the extension $K_n(\gamma)/K_{n-1}(\gamma)$ is \emph{maximal}.  
\end{remark}  
Finally, when $K$ is a number field or function field and $\gamma_{n-1}$ is irreducible over $K$, then Proposition \ref{prop:maximality} and the discriminant formula for $\gamma_n$ in \cite[Proposition 6.2]{Me:LeftRightTotal} imply that $K_n(\gamma)/K_{n-1}(\gamma)$ is maximal if $\gamma_n(0)$ has a primitive prime divisor appearing to odd valuation. However, since we only apply this fact to $K=\mathbb{Q}(t)$ in this paper, we state this maximality criterion for such $K$ only; see \cite[Theorem 6.8]{Me:LeftRightTotal} for a more general statement and proof. In what follows, given an irreducible polynomial $\mathfrak{p}$ in $k[t]$, we let $v_{\mathfrak{p}}: k(t)\rightarrow\mathbb{Z}$ denote its usual valuation. \vspace{.1cm}    
\begin{theorem}\label{thm:GaloisMax} Let $K=k(t)$ for some field $k$ with $\text{char}(k)\neq2$, let $S=\{x^2+c_1,\dots, x^2+c_s\}$ for some $c_i\in k[t]$, and let $\gamma\in \Phi_S$. Moreover for $n\geq2$, assume the following statements hold:\vspace{.1cm}  
\begin{enumerate} 
\item[\textup{(1)}] $\gamma_{n-1}$ is irreducible in $K[x]$.  \vspace{.15cm}
\item[\textup{(2)}] There is an irreducible polynomial $\mathfrak{p}$ in $k[t]$ such that $v_{\mathfrak{p}}(\gamma_m(0))=0$ for all $m<n$ and $v_{\mathfrak{p}}(\gamma_n(0))$ is odd. \vspace{.15cm}
\end{enumerate}
Then the subextension $K_n(\gamma)/K_{n-1}(\gamma)$ is maximal, i.e., $[K_n(\gamma):K_{n-1}(\gamma)]=2^{2^{n-1}}$. \vspace{.1cm} 
\end{theorem} 
For a few more statements about iterated discriminants and extensions generated by sets of unicritical polynomials with a common critical point, $S=\{a(x-c)^d+b\,:a,b\in K, d\geq2\}$,
see \cite[\S6]{Me:LeftRightTotal}. 

\section{Finite-orbit points in the orbit of zero}\label{sec:finiteorbit}
We begin with some notation. Let $S$ be a set of polynomials defined over a field $K$, and let $M_S$ denote the monoid (semigroup plus the identity) generated by $S$ under composition. Then given a point $P$, we call the set $\Orb_S(P)=\{f(P)\,:\,f\in M_S\}$ the \emph{orbit} of $P$ under $S$. In particular, we say that $P$ is a \emph{finite orbit point for $S$} if $\Orb_S(P)$ is a finite set. 

The primary goal of this section is to classify the sets $S=\{x^2+c_1,\dots,x^2+c_s\}$ over $K=\mathbb{Q}$ for which $\Orb_S(0)$ contains a finite orbit point, an obstruction to producing finite index arboreal representations with positive probability in this setting. In particular, as a first step we show that such sets are necessarily defined over the integers. However, since the proof of this fact uses only basic properties of valuations, we state this result in a more general way. In particular, we obtain the amusing corollary that there are no such sets defined over function fields unless all of the $c$'s are defined over the field of constant functions. With this in mind, we begin with the following elementary fact; see also \cite{walde/russo:1994}.
\begin{lem}\label{lem:preper_val}
Let $K$ be a field, let $v$ be a valuation on $K$, and let $d \ge 2$ be an integer. If $\alpha \in K$ is preperiodic for $x^d + c$, then $v(c) < 0$ if and only if $v(\alpha) < 0$.
Moreover, in this case we have $v(c) = dv(\alpha)$.
\end{lem}
\begin{proof}
Let $\phi(x) = x^d + c$. Since $\alpha$ is preperiodic for $\phi$, there exist integers $m < n$ such that $\phi^m(\alpha) = \phi^n(\alpha)$. If $v(c) \ge 0$, then $\phi^n(x) - \phi^m(x)$ is monic with $v$-integral coefficients, so $v(\alpha) \ge 0$ as well. Now suppose that $v(c) < 0$. If $v(\alpha) < \frac{v(c)}{d}$, then
	\[
		v(\phi(\alpha)) = v(\alpha^d + c) = dv(\alpha) < v(\alpha).
	\]
By induction, $v(\phi^n(\alpha)) = d^nv(\alpha) \to -\infty$, so $\alpha$ cannot be preperiodic. On the other hand, if $v(\alpha) > \frac{v(c)}{d}$, then
	\[
		v(\phi(\alpha)) = v(\alpha^d + c) = v(c) < \frac{v(c)}{d}.
	\]
By the previous case, $\phi(\alpha)$ cannot be preperiodic, hence the same is true for $\alpha$.
\end{proof}
In particular, we use the fact above to deduce that if $S$ is a set of polynomials of the form $x^{d_i}+c_i$ and $\Orb_S(0)$ contains a finite orbit point, then the valuation of each $c_i$ must be non-negative. 
\begin{prop}\label{prop:main}
Let $K$ be a field, let $d_1,\ldots,d_s \ge 2$ be integers, and let $c_1,\ldots,c_s \in K$. Let $S = \{x^{d_1} + c_1,\ldots, x^{d_s} + c_s\}$, and suppose that $\mathrm{Orb}_S(0)$ contains a finite orbit point. Then $v(c_1),\ldots, v(c_s) \ge 0$ for every valuation $v$ on $K$.
\end{prop}
\begin{proof}
For each $i = 1,\ldots,s$, let $\phi_i(x) = x^{d_i} + c_i$. Let $\alpha \in \mathrm{Orb}_S(0)$ be a finite orbit point for $S = \{\phi_1,\ldots,\phi_s\}$. Suppose for contradiction that there is some valuation $v$ on $K$ for which at least one of the valuations $v(c_i)$ is negative. Since $\alpha$ is a finite-orbit point, $\alpha$ is preperiodic for each $\phi_i$, so we also have $v(\alpha) < 0$ by Lemma~\ref{lem:preper_val}. (Note that this implies $\alpha \ne 0$.) More precisely, we have
	\[
		v(c_i) = d_i v(\alpha) \text{ for all } i = 1,\ldots,s.
	\]
Now let $\gamma = (\phi_{i_1}, \phi_{i_2}, \ldots )$ be any element of $\Phi_S$. We claim that
	\[
		v(\gamma_n(0)) = d_{i_1} \cdots d_{i_n} v(\alpha)
	\]
for all $n \ge 1$. The conclusion of the proposition now follows from the claim: Indeed, since we assumed $\alpha$ was in the orbit of $0$, we have $\alpha = \gamma_n(0)$ for some $\gamma \in \Phi_S$ and $n \ge 1$. But then $v(\alpha) = d_{i_1}\cdots d_{i_n} v(\alpha)$, contradicting the fact that $v(\alpha) \ne 0$ and $d_i \ge 2$ for all $i = 1,\ldots,s$.

It remains to prove the claim, which we do by induction on $n$. For $n = 1$, we have
	\[
	v(\gamma_1(0)) = v(\phi_{i_1}(0)) = v(c_{i_1}) = d_{i_1}v(\alpha)	
	\]
by Lemma~\ref{lem:preper_val}. Now, for $n > 1$, we write
	\[
	v(\gamma_n(0)) = v\big((\phi_{i_1} \circ \cdots \circ \phi_{i_n})(0)\big)
		= v\big((\phi_{i_2} \circ \cdots \circ \phi_{i_n})(0)^{d_{i_1}} + c_{i_1}\big).
	\]
By our induction hypothesis, we have
	\[
	v\big((\phi_{i_2} \circ \cdots \circ \phi_{i_n})(0)\big) = d_{i_2} \cdots d_{i_n} v(\alpha).
	\]
Since $v(\alpha) < 0$ and $d_i > 2$ for each $i = 1,\ldots,s$, we have
	\[
		v\big((\phi_{i_2} \circ \cdots \circ \phi_{i_n})(0)^{d_{i_1}}\big) = d_{i_1} \cdot d_{i_2} \cdots d_{i_n} v(\alpha) < d_{i_1} v(\alpha) = v(c_{i_1}),
	\]
from which it follows that
	\[
		v(\gamma_n(0)) = \min\left\{v\big((\phi_{i_2} \circ \cdots \circ \phi_{i_n})(0)^{d_{i_1}}\big), v(c_{i_1})\right\}
			= d_{i_1} \cdots d_{i_n} v(\alpha).
	\]
\end{proof}
\vspace{-.5cm} 
In particular, we obtain the following immediate corollary.  
\begin{corollary}\label{cor:int}
Let $S = \{x^{d_1} + c_1, \ldots, x^{d_s} + d_s\}$ for some integers $d_i\ge 2$ and some $c_i \in \overline{\mathbb{Q}}$. If $\mathrm{Orb}_S(0)$ contains a finite-orbit point, then $c_1,\ldots,c_s$ are all algebraic integers.
\end{corollary}
Moreover, we also have the following consequence for function fields. Recall that $K/k$ is a function field if $K$ is a finite extension of $k(t_1,\dots,t_n)$ for some $k$-algebraically independent elements $t_1,\dots, t_n$. Moreover, $n$ is called the transcendence degree of $K$.  
\begin{corollary}
Let $K/k$ be a function field and let $S = \{x^{d_1} + c_1, \ldots, x^{d_s} + d_s\}$ for some integers $d_i\ge 2$ and some $c_i\in K$. If $\mathrm{Orb}_S(0)$ contains a finite-orbit point, then $c_1,\ldots,c_s \in k$.
\end{corollary}
\begin{proof} Suppose that $K/k$ has transcendence degree $1$; the general case follows by induction. Since every place of $K$ is nonarchimedean, Proposition~\ref{prop:main} tells us that $v(c_i) \ge 0$ for every place of $K$ and every $i = 1,\ldots,s$. But by the product formula, this implies that $v(c_i) = 0$ for every place of $K$. Hence, each $c_i$ is a constant.
\end{proof}    
We now return to the problem of classifying the sets $S$ of \emph{quadratic} polynomials of the form $x^2+c_i$ for $c_i\in\mathbb{Q}$ for which $\Orb_S(0)$ contains a finite orbit point. In fact, we will see that the only such sets are:  
\[S=\big\{x^2\big\},\; \big\{x^2-1\big\},\; \big\{x^2-2\big\},\; \big\{x^2,\,x^2-1\big\},\;\big\{x^2-2,\, x^2-3\big\}\;\textup{or}\; \big\{x^2-2,\, x^2-6\big\}.\]

\vspace{-.1cm}       
\begin{remark} It is tempting to think that the classification above is trivial and follows from the fact that the only individual maps $x^2+c$ for $c\in\mathbb{Q}$ where $0$ has finite orbit are $c=0,-1,-2$ (i.e., the PCF maps). However, it is possible for $\Orb_S(P)$ to contain a finite orbit point for a set of quadratic polynomials $S$ \textbf{without} being preperiodic for any of the individual maps in $S$. For an explicit example, consider $S=\{x^2+x,x^2-6x\}$ and $P=2$.    
\end{remark}
In particular, since we now know that the coefficients of the polynomials in such $S$ are integral, we may use the classification of pairs of integral polynomials of the form $x^2+c$ possessing \emph{any} finite orbit point over $\mathbb{Q}$. This result follows from work in \cite[Section 2]{Me:FiniteOrbit}.      
\begin{lemma}\label{lem:intfinite} Let $S=\{x^2+c_1, x^2+c_2\}$ for some distinct $c_i\in\mathbb{Z}$. If $S$ has a finite orbit point $P\in\mathbb{Q}$, then up to reordering $c_1$ and $c_2$, we have that \vspace{.1cm}  
\[(c_1,c_2)=\Big(\frac{1-y^2}{4}, \frac{1-(y+2)^2}{4}\Big)\;\;\;\text{or}\;\;\;(c_1,c_2)=\Big(\frac{1-y^2}{4}, \frac{-3-y^2}{4}\Big) \vspace{.2cm}  \]
for some $y\in\mathbb{Z}$ and $y\equiv{\pm1}\Mod{4}$. 
\end{lemma} 
\begin{proof} Let  $S=\{x^2+c_1, x^2+c_2\}$ for distinct $c_i\in\mathbb{Z}$ and assume that $P\in\mathbb{Q}$ is a finite orbit point for $S$. Then in particular, $P$ is a preperiodic point for both $\phi_1=x^2+c_1$ and $\phi_2=x^2+c_2$. Hence, \cite[Theorem 9]{Morton} and \cite[Exercise 2.20]{SilvBook} together imply that $P$ enters a $1$ or $2$-cycle for both $\phi_1$ and $\phi_2$ (meaning that there is an integer $n_i$ such that $\phi_i^{n_i}(P)$ is a fixed point or a periodic point of exact period $2$ for $\phi_i$). From here, we proceed in cases: \\[5pt] 
\textbf{Case(1):} $P$ enters a fixed point for both maps. In particular, both $\phi_1$ and $\phi_2$ have rational fixed points, and (after replacing $P$ with $\phi_1^{n_1}(P)$ for some $n_1$) we may assume that a fixed point for $\phi_1$ has finite orbit under $S$. Hence, the tuple $(c_1,c_2,P)$ satisfies the hypotheses of \cite[Lemma 2.2]{Me:FiniteOrbit}, and therefore the pair $(c_1,c_2)\in\mathbb{Z}\times\mathbb{Z}$ must be of the form \vspace{.1cm}
\begin{equation}\label{oldfact1} 
(c_1,c_2)=\bigg(\frac{1-y^2}{4}, \frac{1-(y+2)^2}{4}\bigg)\;\;\; \text{or}\;\;\; (c_1,c_2)=\bigg(\frac{t^4-18t^2+1}{4(t^2-1)^2}, \frac{-3t^4-10t^2-3}{4(t^2-1)^2}\bigg)\vspace{.1cm}
\end{equation} 
for some $y,t\in\mathbb{Q}$. However, in the case on the left $y\in\mathbb{Z}$ and $y\equiv{\pm1}\Mod{4}$ since $c_1$ is integral and $\mathbb{Z}\subseteq\mathbb{Q}$ is integrally closed. In particular, we recover the first family in the conclusion of Lemma \ref{lem:intfinite}. On the other hand, when $(c_1,c_2)=(\frac{t^4-18t^2+1}{4(t^2-1)^2}, \frac{-3t^4-10t^2-3}{4(t^2-1)^2})$, let $w=\frac{4t}{t^2-1}$ and $z=\frac{2t^2+2}{t^2-1}$. Then, we see that 
\[c_1=\frac{1-w^2}{4},\;\; c_2=\frac{1-z^2}{4},\;\;\text{and}\;\; w^2-z^2=-4.\] 
In particular, $w$ and $z$ are both integers since $c_1,c_2\in\mathbb{Z}$ and $\mathbb{Z}\subseteq\mathbb{Q}$ is integrally closed. However, it is straightforward to check that the only \emph{integral} solutions to $w^2-z^2=-4$ are $w=0$ and $z=\pm{2}$. But this restriction on $w=4t/(t^2-1)$ forces $t=0$ and $(c_1,c_2)=(1/4,-3/4)$, contradicting our assumption that $c_1$ and $c_2$ are integers. Hence, the only integral pairs of $c$'s in this case are given by $(c_1,c_2)=(\frac{1-y^2}{4}, \frac{1-(y+2)^2}{4})$ for some $y\in\mathbb{Z}$ and $y\equiv{\pm1}\Mod{4}$.            
\\[5pt] 
\textbf{Case(2):} $P$ enters a fixed point for one map and a $2$-cycle for the other. Then, without loss of generality, we may assume that $P$ enters a fixed point for $\phi_1$ and a $2$-cycle for $\phi_2$. In particular, $\phi_1$ has a rational fixed point and $\phi_2$ has a rational point of exact period $2$. Moreover, after replacing $P$ with $\phi_1^{n_1}(P)$ for some $n_1$, we may assume that a fixed point for $\phi_1$ has finite orbit under $S$. Hence, the tuple $(c_1,c_2,P)$ satisfies the hypotheses of \cite[Lemma 2.3]{Me:FiniteOrbit}, and therefore the pair $(c_1,c_2)\in\mathbb{Z}\times\mathbb{Z}$ must be of the form    
\begin{equation}\label{oldfact2} 
(c_1,c_2)=\bigg(\frac{1-y^2}{4}, \frac{-3-y^2}{4}\bigg)\;\;\; \text{or}\;\;\; (c_1,c_2)=\bigg(\frac{-15t^4-2t^2+1}{4(t^2-1)^2}, \frac{-3t^4-10t^2-3}{4(t^2-1)^2}\bigg)\vspace{.1cm}
\end{equation} 
for some $y,t\in\mathbb{Q}$. However, by a similar argument to that given in Case (1), only the left parametrization produces integral $c$-values. Moreover, $y\in\mathbb{Z}$ and $y\equiv{\pm1}\Mod{4}$ in that case.  
\\[5pt]
\textbf{Case(3):} $P$ enters a $2$-cycle both maps. In particular, both $\phi_1$ and $\phi_2$ have rational points of exact period $2$, and (after replacing $P$ with $\phi_1^{n_1}(P)$ for some $n_1$) we may assume that a rational point of exact period $2$ for $\phi_1$ has finite orbit under $S$. Hence, the tuple $(c_1,c_2,P)$ satisfies the hypotheses of \cite[Lemma 2.4]{Me:FiniteOrbit}, and therefore the pair $(c_1,c_2)\in\mathbb{Z}\times\mathbb{Z}$ must be of the form    
\begin{equation}\label{oldfact3} 
(c_1,c_2)=\bigg(\frac{-7t^4-2t^2-7}{4(t^2-1)^2}, \frac{-3t^4-10t^2-3}{4(t^2-1)^2}\bigg)\vspace{.1cm}
\end{equation} 
for some $t\in\mathbb{Q}$. However, by a similar argument to that given in Case (1), one can show that there are no integral $c$-values produced by this parametrization. 
\end{proof}

We also note that if $S=\{x^2+c_1,\dots,x^2+c_s\}$ over the integers has at least 3 polynomials, then there are \emph{no} rational finite orbit points for $S$. This result likely follows from Lemma \ref{lem:intfinite} above, but we simply quote this fact from \cite[Corollary 1.2]{Me:FiniteOrbit}.   
\begin{theorem}\label{thm:nopoints} Let $S=\{x^2 +c_1,x^2 +c_2,\dots,x^2 +c_s\}$ for some distinct $c_i\in\mathbb{Z}$. If $\#S\geq3$, then there are no points $P\in\mathbb{Q}$ with finite orbit for $S$.
\end{theorem} 
Finally, we need the following observation, which roughly says that if $\Orb_S(0)$ contains a finite orbit point, then some pair of coefficients $c_i$ and $c_j$ must be close.   
\begin{lemma}\label{lem:0escapes} Let $S=\{x^2+c_1,\dots,x^2+c_s\}$ for some $c_i\in\mathbb{Z}$. If 
\begin{equation}\label{condition:r=2} 
|c_i^2+c_j|>\max_{1\leq k\leq s}\{|c_k|\}
\end{equation} 
for all $1\leq i,j\leq s$, then $\Orb_S(0)$ cannot contain a finite orbit point for $S$.    
\end{lemma} 
\begin{proof} We begin with some notation. For $n\geq1$ define $M_{S,n}=\{\theta_1\circ\dots\circ \theta_n\,:\, \theta_i\in S\}$ and $M_{S,0}=\{\text{id}\}$. Likewise, let $M_{S,n}(0)=\{f(0)\,:\, f\in M_{S,n}\}$, let $U_n=\max\{|a|\,:\, a\in M_{S,n}(0)\}$, and let $L_n=\min\{|a|\,:\, a\in M_{S,n}(0)\}$. Now assume that \eqref{condition:r=2} holds. We prove that 
\begin{equation}\label{UnLn}
U_n<L_{n+1}\;\text{for $n\geq0$.}
\end{equation} 
by induction. 
Note first that the statement above is true for $n=0$ since \eqref{condition:r=2} implies that none of the $c_i$'s is $0$. Moreover, \eqref{UnLn} is exactly \eqref{condition:r=2} for $n=1$. Now suppose that \eqref{condition:r=2} is true for all $n\leq N$ with $N\geq1$ and let $\ell\in M_{S,N+2}(0)$ be such that $|\ell|=L_{N+2}$ and let $u\in M_{S,N+1}(0)$ be such that $|u|=U_{N+1}$. Next write $\ell=a^2+c_i$ for some $a\in M_{S,N+1}(0)$. Then since $|a|\geq L_{N+1}\geq U_N+1$, we have that
\begin{equation}\label{UnLn2}
|\ell|\geq\ell\geq U_N^2+2U_N+1+c_i\geq U_N^2+U_N+1;
\end{equation}
here the last inequality follows from the fact that $U_N\geq U_1$ by the induction hypothesis (and that $N\geq1$) and that $U_1=\max_{1\leq i\leq s}\{|c_i|\}$. On the other hand, we may write $u=b^2+c_j$ for some $b\in M_{S,N}(0)$. Then, since $b\leq U_N$ and $c_j\leq U_1\leq U_N$, we see that 
\begin{equation}\label{UnLn3}
|u|\leq b^2+|c_j|\leq U_N^2+U_N.
\end{equation} 
In particular, \eqref{UnLn} follows from combining \eqref{UnLn2} and \eqref{UnLn3}. But then $\{L_n\}$ is a strictly increasing sequence of integers. Hence, for all $B$ there exists $m=m(B)$ such that $|F(0)|>B$ for all $F\in M_{S,n}$ with $n\geq m$. This precludes the possibility of $\Orb_S(0)$ containing a finite orbit point: if $g(0)$ is a finite orbit point, then $|f(g(0))|\leq B$ for some $B$ and all $f\in M_S$, a contradiction.            
\end{proof} 
\begin{remark}\label{rem:0escape} In particular, if $S=\{x^2+c_1,x^2+c_2\}$ for $c_i\in\mathbb{Z}$ and 
\[|c_i^2+c_j|>|c_1|+|c_2|\qquad \text{for all $1\leq i,j\leq2$},\] 
then $\Orb_S(0)$ cannot contain a finite orbit point for $S$.      
\end{remark}
We now have all of the tools in place to classify the sets $S=\{x^2+c_1,\dots,x^2+c_s\}$ over the rational numbers for which $\Orb_S(0)$ contains a finite orbit point; this is part (1) of Theorem \ref{thm:obstruction} from the Introduction.   
\begin{theorem}\label{thm:finiteorbit} Let $S=\{x^2+c_1,\dots, x^2+c_s\}$ be a set of quadratic polynomials over $\mathbb{Q}$. If $\Orb_S(0)$ contains a finite orbit point, then $S$ is one of the following exceptional sets: 
\[S=\big\{x^2\big\},\; \big\{x^2-1\big\},\; \big\{x^2-2\big\},\; \big\{x^2,\,x^2-1\big\},\;\big\{x^2-2,\, x^2-3\big\}\;\textup{or}\; \big\{x^2-2,\, x^2-6\big\}. \vspace{.2cm}\] 
\end{theorem} 
\begin{proof} Let $S=\{x^2+c_1,\dots,x^2+c_s\}$ for some $c_i\in\mathbb{Q}$ and suppose that $\Orb_S(0)$ contains a finite orbit point for $S$. Then Corollary \ref{cor:int} implies that each $c_i\in\mathbb{Z}$ and Theorem \ref{thm:nopoints} implies that $\#S\leq2$. If $\#S=1$, then write $S=\{\phi\}$. But in this case, if $\Orb_S(0)$ contains a finite orbit point, then $0$ itself is a finite orbit point for $\phi$. Hence, $\phi$ is a post-critically finite (PCF) map of the form $\phi=x^2+c$ and $c\in\mathbb{Z}$. However, it is well known that the only $c$ with this property are $c=0$,$-1$, and $-2$. That is, $S=\{x^2\}$, $\{x^2-1\}$, and $\{x^2-2\}$ are the only singleton sets (of the desired form) for which $\Orb_S(0)$ contains a finite orbit point. 

It therefore remains to consider the case when $\#S=2$, say $S=\{x^2+c_1,x^2+c_2\}$ for some distinct $c_i\in\mathbb{Z}$. Now, since $\Orb_S(0)$ contains a finite orbit point for $S$, Lemma \ref{lem:intfinite} implies 
\[(c_1,c_2)=\Big(\frac{1-y^2}{4}, \frac{1-(y+2)^2}{4}\Big)\;\;\;\text{or}\;\;\;(c_1,c_2)=\Big(\frac{1-y^2}{4}, \frac{-3-y^2}{4}\Big) \vspace{.2cm}  \]
for some $y\in\mathbb{Z}$ and $y\equiv{\pm1}\Mod{4}$, up to reordering the $c$'s. Suppose first, without loss of generality, that $(c_1,c_2)=(\frac{1-y^2}{4}, \frac{1-(y+2)^2}{4})$. Then, after substituting these expressions in for $c_1$ and $c_2$ into Remark \ref{rem:0escape}, we see that at least one of the following inequalities must hold: \vspace{.15cm} 
\begin{align*} 
\Big|\frac{1}{16}y^4 - \frac{3}{8}y^2 + \frac{5}{16}\Big|&\leq \Big|-\frac{1}{4}y^2 + \frac{1}{4}\Big|+\Big|-\frac{1}{4}y^2 - y - \frac{3}{4}\Big|, \\[5pt]  
\Big|\frac{1}{16}y^4 - \frac{3}{8}y^2 - y - \frac{11}{16}\Big|&\leq  \Big|-\frac{1}{4}y^2 + \frac{1}{4}\Big|+\Big|-\frac{1}{4}y^2 - y - \frac{3}{4}\Big|, \\[5pt] 
\Big|\frac{1}{16}y^4 + \frac{1}{2}y^3 + \frac{9}{8}y^2 + \frac{3}{2}y + \frac{13}{16}\Big|&\leq \Big|-\frac{1}{4}y^2 + \frac{1}{4}\Big|+\Big|-\frac{1}{4}y^2 - y - \frac{3}{4}\Big|, \\[5pt]
\Big|\frac{1}{16}y^4 + \frac{1}{2}y^3 + \frac{9}{8}y^2 + \frac{1}{2}y - \frac{3}{16} \Big|&\leq \Big|-\frac{1}{4}y^2 + \frac{1}{4}\Big|+\Big|-\frac{1}{4}y^2 - y - \frac{3}{4}\Big|. \\  
\end{align*} 
\vspace{-.7cm} 

\noindent But each of these inequalities is true only on some bounded, real interval. Moreover, since we have only a single real parameter $y$, it is a one-variable calculus 
problem to determine each of these intervals. In particular, it is straightforward to check that as a real number $y\in[-6.8,4.7]$, otherwise \emph{all} of the above inequalities fail. On the other hand, $y\in\mathbb{Z}$ and $y\equiv{\pm1}\Mod{4}$ so that $y\in\{-5,-3,-1,1,3\}$. These specific values of $y$ determine the sets $S=\{x^2,x^2-2\}$ and $S=\{x^2-2,x^2-6\}$. Moreover, among these sets, only $S=\{x^2-2,x^2-6\}$ has the desired property that $\Orb_S(0)$ contains a finite orbit point. In this case, $-2\in\Orb_S(0)$ and $-2$ is a finite orbit point for $S$.

Now for the second family from Lemma \ref{lem:intfinite}. Suppose, without loss of generality, that $(c_1,c_2)=(\frac{1-y^2}{4}, \frac{-3-y^2}{4})$ for some $y\in\mathbb{Z}$ and $y\equiv{\pm1}\Mod{4}$. Then, after substituting these expressions in for $c_1$ and $c_2$ into Remark \ref{rem:0escape}, we see that at least one of the following inequalities must hold: \vspace{.15cm} 
\begin{align*} 
\Big|\frac{1}{16}y^4 - \frac{3}{8}y^2 + \frac{5}{16}\Big|&\leq \Big| -\frac{1}{4}y^2 + \frac{1}{4}\Big|+\Big| -\frac{1}{4}y^2 - \frac{3}{4}\Big|, \\[5pt]  
\Big| \frac{1}{16}y^4 - \frac{3}{8}y^2 - \frac{11}{16} \Big|&\leq \Big| -\frac{1}{4}y^2 + \frac{1}{4}\Big|+\Big| -\frac{1}{4}y^2 - \frac{3}{4}\Big|  , \\[5pt] 
\Big| \frac{1}{16}y^4 + \frac{1}{8}y^2 + \frac{13}{16} \Big|&\leq \Big| -\frac{1}{4}y^2 + \frac{1}{4}\Big|+\Big| -\frac{1}{4}y^2 - \frac{3}{4}\Big| , \\[5pt]
\Big| \frac{1}{16}y^4 + \frac{1}{8}y^2 - \frac{3}{16} \Big|&\leq \Big| -\frac{1}{4}y^2 + \frac{1}{4}\Big|+\Big| -\frac{1}{4}y^2 - \frac{3}{4}\Big|. \\  
\end{align*} 
\vspace{-.7cm} 

\noindent But, as before, each of these inequalities is true only on some bounded, real interval. Hence, it is straightforward to check that as a real number $y\in[-4.9,4.9]$, otherwise \emph{all} of the above inequalities fail. On the other hand, $y\in\mathbb{Z}$ and $y\equiv{\pm1}\Mod{4}$ so that $y\in\{-3,-1,1,3\}$. These specific values of $y$ determine the sets          $S=\{x^2,x^2-1\}$ and $S=\{x^2-2,x^2-3\}$. Moreover, both of these sets have the desired property that $\Orb_S(0)$ contains a finite orbit point. In the first case, $0$ is itself a finite orbit point. While in the second case, $-2\in\Orb_S(0)$ is a finite orbit point. This completes the classification in Theorem \ref{thm:finiteorbit}.   
\end{proof} 
\section{Infinite index representations over $\mathbb{Q}$}\label{sec:arboverZ}
We next prove that if $\Orb_S(0)$ contains a finite orbit point, then $S$ produces finite index arboreal representations with probability zero. To do this, we first establish the stability of some relevant sequences. Recall that $M_S$ denotes the semigroup generated by $S$ under composition, that $\nu$ is a strictly positive probability measure on $S$, and that $\bar{\nu}=\nu^{\mathbb{N}}$ is the corresponding product measure on $\Phi_S=S^{\mathbb{N}}$.  
\begin{lemma}\label{lem:specificstable1} Let $S=\{x^2-2,x^2-3\}$, and let $f\in M_S$. Then the following statements hold:  \vspace{.1cm}
\begin{enumerate} 
\item[\textup{(1)}]  If $f(0)\equiv 2\Mod{4}$, then $f$ is Eisenstein at the prime $p=2$. \vspace{.1cm}  
\item[\textup{(2)}]  If $f(0)\equiv\pm{1}\Mod{4}$, then $f(x+1)$ is Eisenstein at the prime $p=2$.  \vspace{.1cm}   
\end{enumerate} 
In particular, every $f\in M_S$ is irreducible over $\mathbb{Q}$.   
\end{lemma} 
\begin{remark} Since $1\in\Orb_S(0)$, proving that every $f\in M_S$ is irreducible over $\mathbb{Q}$ using only Proposition \ref{prop:stability} seems unlikely. Thus the need to use a different technique (in this case, Eisenstein's criterion).    
\end{remark} 
\begin{proof} We begin with some notation. Write $\phi_1=x^2-2$ and $\phi_2=x^2-3$, and let $F=x^2$ and $L=x+1$. Then clearly $\phi_1\equiv F\Mod{2}$, so that $\phi_1$ and $\phi_2$ must commute mod $2$; every polynomial commutes with $F$ mod $2$. In particular, if $f\in M_S$, then we can write $f\equiv \phi_1^n\circ\phi_2^m\Mod{2}$ for some $n,m\geq0$. On the other hand, $\phi_2\equiv (x+1)^2\equiv F\circ L\Mod{2}$ and $L\circ L\equiv x\Mod{2}$. Therefore, every $f\in M_S$ is of the form 
\begin{equation}\label{mod2}  
f\equiv F^n\Mod{2}\qquad \text{or} \qquad f\equiv F^n\circ L\Mod{2}
\end{equation} 
for some $n\geq0$. From here we proceed in cases depending on the congruence class of the constant term of $f$ modulo 4. Note that if $f\in M_S$ is not the identity, then $f(0)\not\equiv0\Mod{4}$, since both $x^2-2$ and $x^2-3$ have no roots modulo $4$.  Hence, we need not consider this case. Suppose first that $f(0)\equiv 2\Mod{4}$. Then, $f(0)\equiv 0\Mod{2}$, and \eqref{mod2} implies that $f\equiv F^n\Mod{2}$. In particular, $f$ satisfies Eisenstein's irreducibility criterion at the prime $p=2$ in this case. On the other hand, if $f(0)\equiv\pm{1}\Mod{4}$, then $f(0)\equiv 1\Mod{2}$ and \eqref{mod2} implies that $f\equiv F^n\circ L\Mod{2}$. Therefore, $f(x+1)\equiv F^n\Mod{2}$. Moreover, the constant term $f(1)$ of $f(x+1)$ is not $0$ mod $4$, again since both $x^2-2$ and $x^2-3$ have no roots in $\mathbb{Z}/4\mathbb{Z}$. Therefore, $f(x+1)$ is Eisenstein at the prime $p=2$ as claimed.              
\end{proof} 
\begin{lemma}\label{lem:specificstable2} Let $S=\{x^2-2,x^2-6\}$. Then every $f\in M_S$ is irreducible over $\mathbb{Q}$.
\end{lemma}
\begin{proof} Let $\phi_1=x^2-2$, let $\phi_2=x^2-6$, and let $F=x^2$. Note that $\phi_1\equiv\phi_2\equiv F\Mod{2}$, and therefore every $f\in M_S$ is of the form $f\equiv F^{n}\Mod{2}$ for some $n\geq0$. Likewise, $\phi_1\equiv\phi_2\Mod{4}$, and hence every $f\in M_S$ is of the form $f\equiv \phi_1^{n}\Mod{4}$ for some $n\geq0$. Moreover, it is straightforward to check that $\phi_1(0)=-2$ and $\phi_1^n(0)=2$ for all $n\geq2$. In particular, $f(0)\equiv\pm{2}\Mod{4}$ for all non-identity $f\in M_S$. Hence, such $f$ are Eisenstein at $p=2$. Therefore, every $f\in M_S$ is irreducible over $\mathbb{Q}$.          
\end{proof} 
We now have the tools in place to prove that if $\Orb_S(0)$ contains a finite orbit point, then $S$ cannot produce finite index arboreal representations with positive probability. 
\begin{remark} The proof of Theorem \ref{thm:obstruction} part (2) below relies on the classification in part (1) in only one way: to ensure that the relevant sets $S$, those for which $\Orb_S(0)$ contains a finite orbit point, also produce stable sequences with probability one. This is likely (by analogy with the case of iterating a single function \cite[Theorem 3.1]{Jones-Survey}) not necessary - $\Orb_S(0)$ containing a finite orbit point should be sufficient to prove infinite index with probability one, without stability assumptions. However, our reliance on stability in the proof of infinite index hinges on our use of Proposition \ref{prop:maximality} above.    
\end{remark} 
\begin{proof}[(Proof of Theorem \ref{thm:obstruction} part \textup{(2)})] Assume that $S$ is one of the exceptional sets in Theorem \ref{thm:obstruction} part (1). If $S=\{x^2\}, \{x^2-1\}, \{x^2-2\}$, or $\{x^2,x^2-1\}$, then it is straightforward to check that the full orbit of $0$ is finite. In particular, if $\gamma\in\Phi_S^{\text{sep}}$ is any sequence of elements of $S$, then the discriminant formula in \cite[Proposition 6.3]{Me:LeftRightTotal} implies that $K_\infty(\gamma)=\bigcup_n K_n(\gamma)$ is a finitely ramified extension. 
Moreover, the same proof of infinite index in \cite[Theorem 3.1]{Jones-Survey} applies in this more general setting: $G_{\gamma,K}$ is (topologically) generated by the conjugacy classes of finitely many elements, and such subgroups of $\Aut(T_\gamma)$ must have infinite index.

In particular, it suffices to consider $S=\{x^2-2,x^2-3\}$ and $S=\{x^2-2,x^2-6\}$. However in both cases, every possible $\gamma_{n}$ is irreducible over $\mathbb{Q}$ for all $\gamma\in\Phi_S$ and all $n\geq0$ by Lemma \ref{lem:specificstable1} and Lemma \ref{lem:specificstable2} respectively. On the other hand, since $\Orb_S(0)$ contains a finite orbit point, there exists a function $f_S\in M_S$ and a finite set $F_S$ such that $g\circ f_S(0)\in F_S$ for all $g\in M_S$; in fact, one can take $f_S=x^2-2$ (for either sets) and $F_S=\{\pm{1},\pm{2}\}$ and $F_S=\{\pm{2}\}$ for $S=\{x^2-2,x^2-3\}$ and $S=\{x^2-2,x^2-6\}$ respectively. With this in mind, consider the set of sequences 
\[N_S:=\big\{\gamma=(\theta_n)_{n\geq1}\in \Phi_S\,:\, \theta_n=x^2-2\;\; \text{i.o.}\big\},\]
whose $n$-th term is $x^2-2$ infinitely often (or with future work in mind, where $\gamma_n=\theta_1\circ\dots \theta_m\circ f_S$ for some $m$ for infinitely many $n$ - that is, the set of sequences $\gamma$ where the function $f_S$ is the tail of $\gamma_n$ infinitely many times). Then it follows from the Borel-Cantelli Theorem (specifically, the Monkey and Typewriter problem \cite[pp. 96-100]{ProbText}) that $\bar{\nu}(N_S)=1$. On the other hand, if $\gamma=(\theta_n)_{n\geq1}\in N_S$ and $\theta_n=x^2-2$, then $\gamma_n(0)\in F_S$. In particular, for each $\gamma\in N_S$ there is a fixed $a_\gamma\in F_S$ such that $\gamma_n(0)=a_\gamma$ for infinitely many $n$ by the Pigeonhole principle. Say $n_1, n_2\dots$ is an infinite (increasing) sequence such that $\gamma_{n_i}(0)=a_\gamma$. Next, we note for all $n\geq2$ the field $K_n(\gamma)$ contains a square root of $\gamma_n(0)$: certainly $K_n(\gamma)$ contains a square root of the discriminant of $\gamma_n$ (since splitting fields always contain a square root of their defining polynomial's discriminant) 
and the discriminant of $\gamma_n$ satisfies: 
\[\disc(\gamma_n)=\Res(\gamma_{n-1},\gamma_{n-1}')^2\cdot 2^{2^n}\cdot\gamma_n(0)\qquad \text{for $n\geq2$};\]
see the proof of the more general discriminant formula in \cite[Proposition 6.2]{Me:LeftRightTotal}. The key point here is that the $\pm{1}$ in \cite[Proposition 6.2]{Me:LeftRightTotal} is $(-1)^{2^{n-1}(2^n-1)}$, which is $+1$ as long as $n\geq2$. In particular, $\sqrt{\gamma_n(0)}\in K_n(\gamma)$ for $n\geq2$ as claimed. Therefore, with the setup above, $\sqrt{a_\gamma}\in K_{n_2}(\gamma)$. However, the fields $K_n(\gamma)$ are nested, and thus $\sqrt{a_\gamma}\in K_{n_2}(\gamma)\subseteq K_{n_i-1}(\gamma)$ for all $i\geq3$. But then Lemma \ref{lem:specificstable1}, Lemma \ref{lem:specificstable2}, and Proposition \ref{prop:maximality} imply that the subextensions $K_{n_i}(\gamma)/K_{n_i-1}(\gamma)$ are not maximal for all $i\geq3$. In particular, the index of $G_{\gamma,\mathbb{Q}}$ in $\Aut(T_\gamma)$ is infinite for all $\gamma\in{N_S}$. Therefore, $G_{\gamma,\mathbb{Q}}$ has infinite index in $\Aut(T_\gamma)$ with probability one as claimed.                      
\end{proof}   

\section{Arboreal Representations over $\mathbb{Z}[t]$}\label{sec:Z[t]}
We now turn our attention to arboreal representations attached to sequences generated by sets over $\mathbb{Z}[t]$. The main advantage in this setting is the abundance of square-free values (and the presence of derivatives and reduction to detect them). More specifically, the main idea is the following (with a few small assumptions): if $\gamma_n(0)$ is square-free, then $\gamma_n(0)$ must contain primitive prime divisors appearing to odd valuation; see Lemma \ref{lem:primdiv} below. In particular, if $\gamma$ is stable and $\gamma_n(0)$ is square-free, then $K_n(\gamma)/K_{n-1}(\gamma)$ is maximal by Theorem \ref{thm:GaloisMax}. However, stability is usually easy to ensure in this setting, and so the main problem becomes how to ensure that $\gamma_n(0)$ is square-free. With this in mind, we have the following convenient trick using derivatives and reduction mod $2$. In what follows, given $c\in\mathbb{Z}[t]$ let $\bar{c}$ denote the polynomial in $\mathbb{F}_2[t]$ obtained by reducing $c$'s coefficients mod $2$. Likewise, given any ring $R$ let $\frac{d}{dt}$ be the usual derivative on the polynomial ring $R[t]$.        

\begin{lemma}\label{lem:square-free} Let $z,c\in\mathbb{Z}[t]$ be such that $z^2+c$ has odd leading term. If $\frac{d}{dt}(\bar{c})=1$ in $\mathbb{F}_2[t]$, then $z^2+c$ is square-free in $\mathbb{Q}[t]$.     
\end{lemma} 
\begin{proof} It suffices to show that $z^2+c$ is square-free in $\mathbb{Z}[t]$ (meaning it has no non-constant square factor) to show it's square-free in $\mathbb{Q}[t]$ by Gauss' Lemma. Suppose for a contradiction that $z^2+c=y^2\cdot w$ for some non-constant $y\in\mathbb{Z}[t]$ and some $w\in\mathbb{Z}[t]$. Note that $y$ must have odd leading term since $z^2+c$ has odd leading term. In particular, the mod $2$ reduction $\bar{y}\in\mathbb{F}_2[t]$ of $y$ must be non-constant. Now we take the expression $z^2+c=y^2\cdot w$, reduce it mod $2$, and take the derivative of both sides in $\mathbb{F}_2[t]$:
\[1=\frac{d}{dt}(\bar{c})=\frac{d}{dt}(\bar{z}^2+\bar{c})=\frac{d}{dt}(\bar{y}^2\cdot\bar{w}\,)=\bar{y}^2\cdot \frac{d}{dt}(\bar{w}).\]
Hence, $\bar{y}^2$ is a unit $\mathbb{F}_2[t]$ and is therefore constant. But this contradicts the previously established fact that $\bar{y}\in\mathbb{F}_2[t]$ is non-constant.  \vspace{.05cm}         
\end{proof}

Next, we have the following elementary bounds for the heights (i.e., degrees) of the points $\gamma_n(0)$ in the critical orbits of sequences in $S$.   \vspace{.05cm}      

\begin{lemma}\label{lem:degree} Let $S=\{x^2+c_1,\dots, x^2+c_s\}$ for some $c_i\in\mathbb{Z}[t]$, let $\gamma=(\theta_n)_{n\geq1}\in\Phi_S$, and assume that $d=\max\{\deg(c_1),\dots,\deg(c_s)\}>0$. Then the following statements hold: \vspace{.2cm}  
\begin{enumerate} 
\item[\textup{(1)}] $\deg(\gamma_n(0))\leq d\cdot 2^{n-1}$ for all $n$.  \vspace{.2cm} 
\item[\textup{(2)}] If $\deg(\theta_n(0))=d$, then $\deg(\gamma_n(0))=d\cdot 2^{n-1}$ and the leading term of $\gamma_n(0)$ is a power of the leading term of $\theta_n(0)$. \vspace{.1cm} 
\end{enumerate}    
\end{lemma} 
\begin{proof} Write $\theta_i(x)=x^2+b_i\in S$ for each $1\leq i\leq n$. Then \vspace{.1cm}
\[\gamma_n(0)=\theta_1(\theta_2(\dots(\theta_{n-1}(\theta_n(0)))))=(((b_n^2+b_{n-1})^2+b_{n-2})^2+\dots b_2)^2+b_1. \vspace{.1cm} \]
Now set $z_0=b_n$ and define 
\begin{equation}\label{sq}
z_{m}=\theta_{n-m}(z_{m-1})=z_{m-1}^2+b_{n-m}
\end{equation} 
recursively for $1\leq m\leq n-1$. Note in particular that $\gamma_n(0)=z_{n-1}$ and it suffices to prove the claim below to prove Lemma \ref{lem:degree}: \\[5pt] 
\textbf{Claim:} $\deg(z_m)\leq d\cdot 2^{m}$ with equality if $\deg(z_0)=d$. Moreover when $\deg(z_0)=d$ and $m\geq1$, the leading term of $z_m$ is the square of the leading term of $z_{m-1}$. We prove this by induction on $m$. The base case $m=0$ is obvious. On the other hand, if $m\geq1$ and the claim holds for $m-1$, then \eqref{sq} implies that 
\[\deg(z_m)\leq\max\{2\deg(z_{m-1}),b_{n-m}\}\leq\{2\cdot d\cdot 2^{m-1},d\}=\max\{d\cdot 2^m,d\}=d\cdot 2^m\]
as desired. Moreover, if $\deg(z_0)=d$ then $\deg(z_{m-1})=d\cdot 2^{m-1}$ by induction. Furthermore, since $2^m\cdot d>d\geq\deg(b_{n-m})$, it follows from \eqref{sq} that $\deg(z_m)=d\cdot 2^{m}$. 
\vspace{.05cm}             
\end{proof}

From here, we combine the previous two lemmas and give a nontrivial criterion for ensuring that the polynomials $\gamma_n(0)$ in the critical orbits of sequences in $S$ are square-free in $\mathbb{Q}[t]$.  \vspace{.05cm}      

\begin{lemma}\label{lem:iteratessquare-free} Let $S=\{x^2+c_1,\dots, x^2+c_s\}$ for some $c_i\in\mathbb{Z}[t]$, let $\gamma=(\theta_n)_{n\geq1}\in\Phi_S$, and let $d=\max\{\deg(c_1),\dots,\deg(c_s)\}>0$. Moreover, assume $\theta_1$ satisfies $\frac{d}{dt}(\overline{\theta_1(0)})=1$ in $\mathbb{F}_2[t]$. Then the following statements hold: \vspace{.15cm}
\begin{enumerate} 
\item[\textup{(1)}] $\pm{\gamma_n(0)}$ is not a square in $\mathbb{Q}[t]$ for all $n\geq1$.  \vspace{.2cm}  
\item[\textup{(2)}] If $\deg(\theta_n(0)))=d$ and $\theta_n(0)$ has odd leading term, then $\gamma_n(0)$ is square-free in $\mathbb{Q}[t]$.   
\end{enumerate}      
\end{lemma}  
\begin{proof} Let $\theta_1(0)=c$ and write $\gamma_n(0)=z^2+c$ where $z=0$ if $n=1$ and $z=\gamma_{n-1}(0)$ if $n\geq2$. For statement (1), suppose $\pm{\gamma_n(0)}$ is a square in $\mathbb{Q}[t]$. Then it follows from Gauss' Lemma that $\pm{\gamma_n(0)}=a\cdot y^2$ for some $a\in\mathbb{Z}$ and some $y\in\mathbb{Z}[t]$. But then 
\[1=\frac{d}{dt}(\bar{z}^2+\bar{c})=\frac{d}{dt}(\overline{\gamma_n(0)})=\frac{d}{dt}(\overline{\pm\gamma_n(0)})=\frac{d}{dt}(\overline{a\cdot y^2})=a\frac{d}{dt}(\bar{y}^2)=0,\]
a contradiction. Therefore, $\pm{\gamma_n(0)}$ is not a square in $\mathbb{Q}[t]$ for all $n\geq1$.   
 
As for the second statement, note that Lemma \ref{lem:degree} implies that the leading term of $\gamma_n(0)$ is odd: it's a power of the odd leading term of $\theta_n(0)$. Hence, Lemma \ref{lem:square-free} applied to $\gamma_n(0)=z^2+c$ implies that $\gamma_n(0)$ is square-free in $\mathbb{Q}[t]$ as claimed. \vspace{.1cm}        
\end{proof}

Next, with Lemma \ref{lem:iteratessquare-free} in mind, we show that $\gamma_n(0)$ has a primitive prime divisor appearing to odd valuation whenever $\gamma_n(0)$ is square-free (subject also to a basic degree condition). \vspace{.1cm}  

\begin{lemma}\label{lem:primdiv}  Let $S=\{x^2+c_1,\dots, x^2+c_s\}$ for some $c_i\in\mathbb{Z}[t]$, let $\gamma=(\theta_n)_{n\geq1}\in \Phi_S$, let $n\geq 2$, and assume the following conditions hold: \vspace{.05cm}
\begin{enumerate} 
\item[\textup{(1)}] $\deg(\theta_n(0))=\max\{\deg(c_1), \dots, \deg(c_s)\}=d>0$, \vspace{.15cm}  
\item[\textup{(2)}] $\gamma_n(0)$ is square-free in $\mathbb{Q}[t]$, 
\end{enumerate} 
Then there exists an irreducible polynomial $p\in\mathbb{Q}[t]$ such that $v_p(\gamma_n(0))=1$ and $v_p(\gamma_m(0))=0$ for all $1\leq m<n$.     
\end{lemma} 
\begin{proof} Since $\gamma_n(0)$ is square-free, $\gamma_n(0)=p_1\cdots p_t$ for some coprime, irreducible polynomials $p_i\in\mathbb{Q}[t]$. In particular, if the conclusion of Lemma \ref{lem:primdiv} is false, then it must be the case that each $p_i\big\vert\gamma_{r_i}(0)$ for some $1\leq r_i\leq n-1$. Therefore, $\gamma_n(0)\big\vert \gamma_1(0)\cdots \gamma_{n-1}(0)$.
However, it then follows from Lemma \ref{lem:degree} that $d\cdot2^{n-1}\leq d\cdot(2^{n-1}-1)$, a contradiction.  
\end{proof}

Finally, putting Lemmas \ref{lem:iteratessquare-free} and \ref{lem:primdiv} together with Theorem \ref{thm:GaloisMax}, we obtain the maximality criterion for sets from the Introduction (which we restate for convenience).  
\thmtool*
\begin{proof} Suppose $\gamma=(\theta_n)_{n\geq1}\in\Phi_S$ satisfies $\theta_1=\phi_j$ and that a particular index $n$ satisfies $\theta_n=\phi_k$. Then first, since $\theta_1(0)=c_j$ and $\frac{d}{dt}(\overline{c_j})=1$ in $\mathbb{F}_2[t]$, Lemma \ref{lem:iteratessquare-free} part 1 implies that $\pm{\gamma_m(0)}$ is not a square in $\mathbb{Q}[t]$ for all $m\geq1$. In particular, $\gamma$ is stable over $\mathbb{Q}(t)$ by Proposition \ref{prop:stability}. Here we use also that if $a\in\mathbb{Q}[t]$ is a square in $\mathbb{Q}(t)$, then it is a square in $\mathbb{Q}[t]$; this follows from the fact that $\mathbb{Q}[t]$ is integrally closed in $\mathbb{Q}(t)$. 

As for the claim about maximality, note that Lemma \ref{lem:iteratessquare-free} part (2) (and the assumed conditions in the theorem) imply that $\gamma_n(0)$ is square-free in $\mathbb{Q}[t]$. In particular, it follows from Lemma \ref{lem:primdiv} that there exists an irreducible polynomial $p\in\mathbb{Q}[t]$ such that $v_p(\gamma_n(0))=1$ and $v_p(\gamma_m(0))=0$ for all $1\leq m<n$ (normally called a primitive prime divisor of $\gamma_n(0)$ appearing to odd valuation). Finally, Theorem \ref{thm:GaloisMax} implies that $K_n(\gamma)/K_{n-1}(\gamma)$ is maximal; here we use that the places of $K=\mathbb{Q}(t)$ correspond to irreducible polynomials in $\mathbb{Q}[t]$ with the usual valuations. 
\end{proof} 

In particular, Theorem \ref{thm:tool} gives us a way to produce sets $S$ with a positive proportion of sequences yielding maximal subextensions infinitely often. \vspace{.05cm}   

\begin{corollary}\label{cor:i.o.} Let $S=\{x^2+c_1,\dots, x^2+c_s\}$ for some polynomials $c_i\in \mathbb{Z}[t]$, let $\nu$ be a strictly positive probability measure on $S$, and let $\bar{\nu}=\nu^{\mathbb{N}}$ be the product measure on $\Phi_S$. Moreover, assume that the following conditions hold:\vspace{.1cm}
\begin{enumerate} 
\item[\textup{(1)}] Some $c_j$ satisfies $\frac{d}{dt}(\overline{c_j})=1$ in $\mathbb{F}_2[t]$. \vspace{.1cm} 
\item[\textup{(2)}] Some $c_k$ satisfies $\deg(c_k)=\max\{\deg(c_1), \dots\deg(c_s)\}$ and $c_k$ has odd leading term. \vspace{.1cm} 
\end{enumerate} 
Then the set of $\mathbb{Q}(t)$-stable sequences $\gamma\in\Phi_S$ defining maximal subextensions $K_n(\gamma)/K_{n-1}(\gamma)$ i.o. has positive measure: 
\[\bar{\nu}\Big(\Big\{\gamma\in\Phi_S\,:\,\text{$\gamma$ is stable over $\mathbb{Q}(t)$ and $K_n(\gamma)/K_{n-1}(\gamma)$ is maximal i.o.}\Big\}\Big)>0.\]
\end{corollary} 
\begin{proof}
Let $\phi_i=x^2+c_i$ for each $c_i$. Then by Theorem \ref{thm:tool}, it suffices to show that 
\[\mathcal{G}:=\Big\{\gamma=(\theta_n)_{n\geq1}\,:\,\text{$\theta_1=\phi_j$ and $\theta_n=\phi_k$ i.o.} \Big\}\]
has positive $\bar{\nu}$ measure to prove Corollary \ref{cor:i.o.}. However, note that $\mathcal{G}=
\mathcal{G}_1\cap\mathcal{G}_2$ where 
\[\mathcal{G}_1:=\Big\{\gamma=(\theta_n)_{n\geq1}\,:\,\text{$\theta_1=\phi_j$}\Big\}\;\;\;\;\text{and}\;\;\;\;\mathcal{G}_2:=\Big\{\gamma=(\theta_n)_{n\geq1}\,:\,\text{$\theta_n=\phi_k$ i.o.}\Big\}\] 
respectively. On the other hand, $\bar{\nu}(\mathcal{G}_1)=\nu(\phi_j)>0$ since $\bar{\nu}$ is the product measure of a countable number of copies of $\nu$ on $S$; see \cite[Theorem 10.4]{Jacod&Protter}. Moreover, since $\nu(\phi_k)>0$, the Borel-Cantelli Theorem (specifically, the Monkey and Typewriter problem \cite[pp. 96-100]{ProbText}) implies that $\bar{\nu}(\mathcal{G}_2)=1$. Therefore, 
\[\bar{\nu}(\mathcal{G})=\bar{\nu}(\mathcal{G}_1\cap\mathcal{G}_2)=\bar{\nu}(\mathcal{G}_1)=\nu(\phi_j)>0\] 
as desired.     
\end{proof}

Finally, we are ready to prove Theorem \ref{thm:nonmonic} from the Introduction (which we restate for convenience below). Intuitively, this result shows that if $S=\{x^2+c_1,\dots, x^2+c_s\}$ for some polynomials $c_i\in \mathbb{Z}[t]$ of fixed degree at most $d$, then as the cardinality of $S$ grows, most such sets produce big arboreal representations over $K=\mathbb{Q}(t)$ with positive probability; recall that big representations are those that are stable and also maximal infinitely often. \vspace{.05cm}  

\thmnonmonic*
\begin{remark}\label{rem:PdB} In what follows, we let \textbf{$x_B=(2B+1)$} to ease notation. Hence, $\#P_d(B)=x_B^{d+1}$ for all positive $B\in\mathbb{Z}$ since there are $2B+1$ integers $\leq B$ and there are $d+1$ possible terms in a polynomial of degree at most $d$. In particular, $\mathcal{S}(d,s,B)$=$\binom{x_B^{d+1}}{s}$ by definition of $\mathcal{S}(d,s,B)$.  
\end{remark} 
\begin{remark}\label{rem:binompoly} Recall that $\binom{X}{s}=\frac{X(X-1)\cdots(X-(s-1))}{s!}$. Hence if we fix $s$ and vary $X$, the binomial coefficient $\binom{X}{s}$ is a polynomial of degree $s$ in $X$ with leading coefficient $(s!)^{-1}$.   
\end{remark} 
\begin{proof} We proceed in cases depending on whether $d$ is even or odd. The key distinction is that in the even case we may choose a single polynomial satisfying conditions (1) and (2) of Theorem \ref{thm:tool}. This is not possible in the odd case, unless $d=1$.  \\[5pt] 
\textbf{Case(1):} Assume that $d$ is even. We begin by counting polynomials $c\in P_d(B)$ satisfying $\frac{d}{dt}(\bar{c})=1$ in $\mathbb{F}_2[t]$, $\deg(c)=d$, and having odd leading term (which we call property \eqref{deven}). This is clearly true of $c$ if and only if:
\begin{equation}\label{deven} 
c=\sum_{i=0}^d a_it^i,\;\,\text{$a_d$ is odd,}\; \text{$a_1$ is odd},\;\text{and}\;\text{$a_i$ is even for all odd indices $3\leq i\leq d-1$}.\tag{*} 
\end{equation} 
In particular, after counting even or odd integers with absolute value at most $B$ for each of the $d/2+1$ stipulated coefficients, we see that \vspace{.1cm}
\begin{equation}\label{deven:count1}
\#\{c\in P_d(B)\,:\, \text{$c$ satisfies property \eqref{deven}}\}=r_d\cdot x_B^{d+1}+O({x_B}^{d}). \vspace{.1cm} 
\end{equation}
Here we use that $r_d= (\frac{1}{2}\big)^{\frac{d}{2}+1}$ for even $d$. 
Hence, it follows from Remark \ref{rem:PdB} that there are $(1-r_d)\cdot {x_B}^{d+1}+O({x_B}^{d})$ polynomials in $P_d(B)$ that do not satisfy property \eqref{deven}. In particular, there are 
\[\binom{(1-r_d)\cdot x_B^{d+1}+O({x_B}^{d})}{s} \vspace{.1cm}\]
sets in $\mathcal{S}(d,s,B)$ whose elements \textbf{all} fail to satisfy property \eqref{deven}. Therefore, there are \vspace{.1cm} 
\begin{equation}\label{deven:count2}
\binom{x_B^{d+1}}{s}-\binom{(1-r_d)\cdot x_B^{d+1}+O(x_B^{d})}{s}=\frac{(1-(1-r_d)^s)}{s!}x_B^{(d+1)s}+O(x_B^{(d+1)s-1}) \vspace{.1cm}
\end{equation} 
sets in $\mathcal{S}(d,s,B)$ with at least one element satisfying property \eqref{deven}. Here we use also that the binomial coefficient $\binom{X}{s}$ is a polynomial of degree $s$ in $X$ with leading coefficient $(s!)^{-1}$; see Remark \ref{rem:binompoly}.  

On the other hand, Corollary \ref{cor:i.o.} implies that such sets $\{c_1,\dots,c_s\}\in \mathcal{S}(d,s,B)$ determine sets of quadratic polynomials $S=\{x^2+c_1,\dots,x^2+c_s\}$ which furnish big arboreal representations over $K=\mathbb{Q}(t)$ with positive probability. Hence, \eqref{deven:count2} and Corollary \ref{cor:i.o.} together imply that \vspace{.1cm}   
\[
\#\Big\{S\in\mathcal{S}(d,s,B)\,:\,\bar{\nu}_S\big(\BigArb(S,K)\big)>0\Big\}\geq \frac{(1-(1-r_d)^s)}{s!}x_B^{(d+1)s}+O(x_B^{(d+1)s-1}). \vspace{.1cm}
\]
In particular, after dividing the inequality above by \vspace{.1cm} 
\[\#S(d,s,B)=\binom{{x_B}^{d+1}}{s}=\frac{1}{s!}\cdot{x_B}^{(d+1)s}+O(x_B^{(d+1)s-1}) \vspace{.1cm} \] 
and letting $B\rightarrow\infty$, we see that \vspace{.1cm}  
\[\liminf_{B\rightarrow\infty}\frac{\#\Big\{S\in\mathcal{S}(d,s,B)\,:\,\bar{\nu}_S\big(\BigArb(S,K)\big)>0\Big\}}{\#\mathcal{S}(d,s,B)}\geq 1-(1-r_d)^s. \vspace{.1cm} \]
Therefore, statement (1) of Theorem \ref{thm:nonmonic} holds for even $d$ and fixed $s$ as claimed. \\[5pt] 
Before we begin the proof of the case when $d$ is odd, we need the following complementary form of the inclusion-exclusion principle.   
\begin{remark}\label{rem:inout} Let $\mathcal{A}$ and $\mathcal{B}$ be subsets of a finite set $\mathcal{X}$ with complements $\mathcal{A}^c$ and $\mathcal{B}^c$ in $\mathcal{X}$ respectively. Then $\#(\mathcal{A}\cap \mathcal{B})=\#\mathcal{X}-\#\mathcal{A}^c-\#\mathcal{B}^c+\#(\mathcal{A}^c\cap \mathcal{B}^c)$; this follows directly from De Morgan's Laws and the usual inclusion-exclusion principle.    
\end{remark} 
\noindent\textbf{Case(2):} Assume that $d$ is odd. We begin by counting polynomials $c\in P_d(B)$ satisfying $\frac{d}{dt}(\bar{c})=1$ in $\mathbb{F}_2[t]$ and call this property \eqref{dodd}. This is clearly true of $c$ if and only if: \vspace{.1cm} 
\begin{equation}\label{dodd} 
c=\sum_{i=0}^d a_it^i,\;\, \text{$a_1$ is odd,}\;\text{and}\;\text{$a_i$ is even for all odd indices $3\leq i\leq d$}.\tag{\RomanNumeralCaps{1}} \vspace{.1cm} 
\end{equation} 
In particular, after counting even or odd integers with absolute value at most $B$ for each of the $(d+1)/2$ stipulated coefficients, we see that \vspace{.1cm}
\begin{equation}\label{dodd:count1}
\#\{c\in P_d(B)\,:\, \text{$c$ satisfies property \eqref{dodd}}\}=r_d\cdot x_B^{d+1}+O(x_B^{d}). \vspace{.1cm} 
\end{equation}
Here we use that $r_d= (\frac{1}{2}\big)^{\frac{d+1}{2}}$ for odd $d$. On the other hand, we say that $c\in P_d(B)$ has property \eqref{dodd:lt} if it has degree $d$ and odd leading term: 
\begin{equation}\label{dodd:lt}
c=\sum_{i=0}^d a_it^d\;\,\text{and $a_d$ is odd.} \tag{\RomanNumeralCaps{2}}
\end{equation}  
Clearly properties \eqref{dodd} and \eqref{dodd:lt} are \textbf{disjoint events} (i.e., there are no polynomials satisfying both properties) by examining the leading coefficient $a_d$ alone, and
\begin{equation}\label{dodd:count2}
\#\{c\in P_d(B)\,:\, \text{$c$ satisfies property \eqref{dodd:lt}}\}=\frac{1}{2}\cdot x_B^{d+1}+O(x_B^{d})
\end{equation}  
polynomials have property \eqref{dodd:lt}. 

Now, with Corollary \ref{cor:i.o.} and big arboreal representations in mind, we count sets in $\mathcal{S}(d,s,B)$ with at least one element satisfying property \eqref{dodd} and at least one element satisfying property \eqref{dodd:lt}. Do do this, let $\mathcal{A}\subseteq \mathcal{S}(d,s,B)$ be the sets in $\mathcal{S}(d,s,B)$ with at least one element satisfying property \eqref{dodd}, and let $\mathcal{B}\subseteq \mathcal{S}(d,s,B)$ be the sets in $\mathcal{S}(d,s,B)$ with at least one element satisfying property \eqref{dodd:lt}; that is, we want to count $\#\mathcal{A}\cap \mathcal{B}$. To do this, note first that \eqref{dodd:count1} and Remark \ref{rem:PdB} imply that there are
\[\# P_d(B)-\Big(r_d\cdot x_B^{d+1}+O(x_B^{d})\Big)=(1-r_d)\cdot x_B^{d+1}+O(x_B^{d})\]
polynomials in $P_d(B)$ that do not have property \eqref{dodd}. Therefore, there are \vspace{.05cm}
\[\binom{(1-r_d)\cdot x_B^{d+1}+O(x_B^d)}{s} \vspace{.05cm}\] 
sets in $\mathcal{S}(d,s,B)$ whose elements \textbf{all} fail to satisfy property \eqref{dodd}. Equivalently, \vspace{.05cm}
\begin{equation}\label{eq:Acomp}
\#\mathcal{A}^c=\binom{(1-r_d)\cdot x_B^{d+1}+O(x_B^d)}{s}; \vspace{.05cm}
\end{equation}
here $\mathcal{A}^c$ denotes the complement of $\mathcal{A}$ in $\mathcal{S}(d,s,B)$. Likewise, \eqref{dodd:count2} and Remark \ref{rem:PdB} imply that there are \vspace{.05cm}
\[\# P_d(B)-\Big(\frac{1}{2}\cdot x_B^{d+1}+O(x_B^{d})\Big)=\frac{1}{2}\cdot x_B^{d+1}+O(x_B^{d}) \vspace{.05cm}\]
polynomials in $P_d(B)$ that do not have property \eqref{dodd:lt}. Therefore, there are \vspace{.05cm}
\[\binom{\frac{1}{2}\cdot x_B^{d+1}+O(x_B^d)}{s} \vspace{.05cm}\] 
sets in $\mathcal{S}(d,s,B)$ whose elements \textbf{all} fail to satisfy property \eqref{dodd:lt}; Equivalently,  \vspace{.05cm}
\begin{equation}\label{eq:Bcomp}
\#\mathcal{B}^c=\binom{\frac{1}{2}\cdot x_B^{d+1}+O(x_B^d)}{s}. \vspace{.05cm} 
\end{equation}
Finally, before we can put all of the pieces above together, it remains to count $\#(\mathcal{A}^c\cap \mathcal{B}^c)$. However, $\mathcal{A}^c\cap \mathcal{B}^c$ consists precisely of the sets in $\mathcal{S}(d,s,B)$ whose elements all fail to satisfy \eqref{dodd} and all fail to satisfy \eqref{dodd:lt}. But, since \eqref{dodd} and \eqref{dodd:lt} are disjoint events, there are exactly \vspace{.05cm}   
\[\# P_d(B)-\Big(r_d\cdot x_B^{d+1}+O(x_B^{d})\Big)-\Big(\frac{1}{2}\cdot x_B^{d+1}+O(x_B^{d})\Big)=\big(1-r_d-\frac{1}{2}\big)\cdot x_B^{d+1}+O(x_B^{d}) \vspace{.05cm}\]
polynomials that fail to satisfy \eqref{dodd} and fail to satisfy \eqref{dodd:lt}; 
here we use Remark \ref{rem:PdB}, \eqref{dodd:count1}, and \eqref{dodd:count2}. Therefore, we see that \vspace{.05cm}
\begin{equation}\label{eq:ABcomp}
\#(\mathcal{A}^c\cap \mathcal{B}^c)=\binom{(1-r_d-\frac{1}{2})\cdot x_B^{d+1}+O(x_B^d)}{s}. \vspace{.05cm}
\end{equation}   
In particular, \eqref{eq:Acomp}, \eqref{eq:Bcomp}, \eqref{eq:ABcomp}, and the complementary form of the inclusion-exclusion principle in Remark \ref{rem:inout} applied to $\mathcal{X}=\mathcal{S}(d,s,B)$, $\mathcal{A}$ and $\mathcal{B}$ above imply that there are \vspace{.15cm}  
\[
\Scale[1.3]{
\binom{x_B^{d+1}}{s}-\binom{(1-r_d)\cdot x_B^{d+1}+O(x_B^d)}{s}-\binom{\frac{1}{2}\cdot x_B^{d+1}+O(x_B^d)}{s}+\binom{(1-r_d-\frac{1}{2})\cdot x_B^{d+1}+O(x_B^d)}{s}} \vspace{.15cm} 
\]
sets in $\mathcal{S}(d,s,B)$ with at least one element satisfying property \eqref{dodd} and at least one element satisfying property \eqref{dodd:lt} (i.e., the expression above is $\#\mathcal{A}\cap\mathcal{B}$). Hence, there are \vspace{.15cm} 
\begin{equation}\label{dodd:big} 
\frac{1-(1-r_d)^s-(\frac{1}{2})^s+(1-r_d-\frac{1}{2})^s}{s!}x_B^{(d+1)s}+O(x_B^{(d+1)s-1}) \vspace{.155cm} 
\end{equation} 
sets in $\mathcal{S}(d,s,B)$ with at least one element satisfying property \eqref{dodd} and at least one element satisfying property \eqref{dodd:lt}; see also Remark \ref{rem:binompoly}. On the other hand, Corollary \ref{cor:i.o.} implies that such sets $\{c_1,\dots,c_s\}\in\mathcal{S}(d,s,B)$ determine sets of quadratic polynomials $S=\{x^2+c_1,\dots,x^2+c_s\}$ which furnish big arboreal representations over $K=\mathbb{Q}(t)$ with positive probability. Hence, \eqref{dodd:big} and Corollary \ref{cor:i.o.} together imply that \vspace{.15cm}   
\[
\Scale[.999]{
\#\Big\{S\in\mathcal{S}(d,s,B)\,:\,\bar{\nu}_S\big(\BigArb(S,K)\big)>0\Big\}\geq \frac{1-(1-r_d)^s-(\frac{1}{2})^s+(1-r_d-\frac{1}{2})^s}{s!}x_B^{(d+1)s}+O(x_B^{(d+1)s-1})} \vspace{.15cm}
\]
when $d$ is odd. In particular, after dividing the inequality above by \vspace{.1cm} 
\[\#S(d,s,B)=\binom{{x_B}^{d+1}}{s}=\frac{1}{s!}\cdot{x_B}^{(d+1)s}+O(x_B^{(d+1)s-1}) \vspace{.1cm} \] 
and letting $B\rightarrow\infty$, we see that \vspace{.15cm}
\[\displaystyle{\liminf_{B\rightarrow\infty}}\frac{\#\Big\{S\in \mathcal{S}(d,s,B)\,:\,\bar{\nu}_S\big(\BigArb(S,K)\big)>0\Big\}}{\#\mathcal{S}(d,s,B)}\geq 1-(1-r_{d})^s-\Big(\frac{1}{2}\Big)^s+\Big(1-r_d-\frac{1}{2}\Big)^s \vspace{.15cm}\]
for odd $d$ as claimed in part (2) of Theorem \ref{thm:nonmonic}.               
\end{proof}

On the other hand, if we restrict ourselves to monic polynomials of fixed even degree, then we can we can improve Theorem \ref{thm:nonmonic} and count sets producing surjective arboreal representations with positive probability, instead of just big arboreal representations with positive probability. With this in mind, we fix a bit more notation: let $M_d(B)$ be the set of monic, degree-$d$ polynomials in $\mathbb{Z}[t]$ whose coefficients all have absolute value at most $B$, and let    
\[\mathcal{S}^{\textup{Mon}}(d,s,B):\big\{\{c_1,c_2,\dots, c_s\}\,:\, c_i\in M_d(B)\big\}\]
be the set of $s$-element sets in $M_d(B)$. Moreover, given an element $\{c_1,\dots,c_s\}\in \mathcal{S}^{\textup{Mon}}(d,s,B)$ we associate a set of quadratic polynomials with coefficients in $\mathbb{Z}[t]$, \vspace{.05cm}  
\[ S=S\big(\{c_1,\dots,c_s\}\big)=\{x^2+c_1,\dots, x^2+c_s\}, \vspace{.05cm} \] 
and study the sequences in $S$ furnishing surjective arboreal representation over $K=\mathbb{Q}(t)$:  
\[\SurArb(S):=\Big\{\gamma\in\Phi_S^{\textup{sep}}\,:\, G_{\gamma,\,\mathbb{Q}(t)}=\Aut(T_\gamma)\Big\}.\]
In particular, we prove that most sets in $\mathcal{S}^{\textup{Mon}}(d,s,B)$ produce surjective representations with positive probability as $s$ grows. 
\begin{theorem}\label{thm:monic} Let $d>0$, let $s\geq2$, and let $\mathcal{S}^{\textup{Mon}}(d,s,B)$ and $\SurArb(S)$ be as above. If $d$ is even, then  \vspace{.1cm}
\[\liminf_{B\rightarrow\infty}\frac{\#\Big\{S\in \mathcal{S}^{\textup{Mon}}(d,s,B)\,:\,\bar{\nu}_S\big(\SurArb(S)\big)>0\Big\}}{\#\mathcal{S}^{\textup{Mon}}(d,s,B)}\geq 1-\Bigg(1-\Big(\frac{1}{2}\Big)^{\frac{d}{2}}\Bigg)^{\hspace{-.1cm}s}. \vspace{.1cm}\]  
In particular, when $d\geq2$ is fixed and even, almost all sets $S=\{x^2+c_1,\dots,x^2+c_s\}$ with $c_i\in M_d$ furnish surjective arboreal representations with positive probability as $s\rightarrow\infty$. 
\end{theorem}
\begin{proof} The proof is very similar to that of statement (1) of Theorem \ref{thm:nonmonic} (only now the leading terms are fixed). In particular, it is straightforward to check that $\#M_d(B)=x_B^d$ and
\begin{equation}  
\#\Big\{c\in M_d(B)\,:\,\text{$\frac{d}{dt}(\bar{c})=1$ in $\mathbb{F}_2[t]$}\Big\}=\Big(\frac{1}{2}\Big)^{\frac{d}{2}}\cdot x_B^d +O(x_B^{d-1}), 
\end{equation}  
since the $c$'s above must have even coefficients at every odd-powered term, except for the linear term, which must be odd - that is, there is a parity stipulation on the coefficients of $c$ at every odd-powered term (and $d/2$ such terms). Therefore, similar to Case (1) in the proof of Theorem \ref{thm:nonmonic}, there are \vspace{.1cm} 
\begin{equation}\label{eq:count-monic}
\binom{x_B^d}{s}-\binom{\big(1-({\frac{1}{2}})^{\frac{d}{2}}\big)\cdot x_B^{d}+O({x_B}^{d-1})}{s}=\frac{\Big(1-\big(1-({\frac{1}{2}})^{\frac{d}{2}}\big)^s\Big)}{s!}\,x_B^{ds}+O(x_B^{ds-1}) \vspace{.15cm}
\end{equation} 
sets in $\mathcal{S}^{\text{Mon}}(d,s,B)$ with at least one term $c_j\in\mathbb{Z}[t]$ satisfying $\frac{d}{dt}(\bar{c}_j)=1$ in $\mathbb{F}_2[t]$. 	On the other hand, such collections determine sets of quadratic polynomials $S=\{x^2+c_1,\dots,x^2+c_s\}$ which furnish surjective arboreal representations over $\mathbb{Q}(t)$ with positive probability by Theorem \ref{thm:tool}: the sequences $\gamma=(\theta_n)_{n\geq1}$ with $\theta_1=x^2+c_j$ produce maximal extensions for every $n$ since condition (2) of Theorem \ref{thm:tool} is satisfied by all $c\in M_d(B)$. Therefore, it follows from \eqref{eq:count-monic} that \vspace{.1cm}   
\begin{equation}\label{eq-count-monic2}
\#\Big\{S\in\mathcal{S}^{\text{Mon}}(d,s,B) \,:\,\bar{\nu}_S\big(\SurArb(S)\big)>0\Big\}\geq \frac{\Big(1-\big(1-({\frac{1}{2}})^{\frac{d}{2}}\big)^s\Big)}{s!}\,x_B^{ds}+O(x_B^{ds-1}). \vspace{.1cm} 
\end{equation} 
In particular, after dividing both sides of \eqref{eq-count-monic2} by \vspace{.05cm}
\[\#\mathcal{S}^{\text{Mon}}(d,s,B)=\binom{x_B^d}{s}=\frac{1}{s!}x_B^{ds}+O(x_B^{ds-1})\vspace{.05cm}\]  
and letting $B\rightarrow\infty$, we obtain the lower bound \vspace{.15cm}  
\[\liminf_{B\rightarrow\infty}\frac{\#\Big\{S\in \mathcal{S}^{\textup{Mon}}(d,s,B)\,:\,\bar{\nu}_S\big(\SurArb(S)\big)>0\Big\}}{\#\mathcal{S}^{\textup{Mon}}(d,s,B)}\geq 1-\Bigg(1-\Big(\frac{1}{2}\Big)^{\frac{d}{2}}\Bigg)^{\hspace{-.1cm}s}\vspace{.15cm}\]
in Theorem \ref{thm:monic} as claimed.                
\end{proof} 
\section{Applications to density of prime divisors}\label{sec:density}

Let $K$ be a number field with ring of integers $\mathcal{O}_K$ and let $\gamma=(\theta_n)_{n\geq1}$ for some $\theta_n\in K[x]$. In this section we consider sequences of the form $(\gamma_n(a_0))_{n \geq 0}$, where $a_0 \in K$, $\gamma_0(x) = x$, and $\gamma_n(x) = (\theta_1 \circ \cdots \circ \theta_n)(x)$ for $n \geq 1$. In particular, we are interested in the set of prime ideal divisors of our sequence, namely
$$
P(\gamma, a_0) := \{\p \subset \O_K : \text{$\p$ is prime and $\p \mid \gamma_n(a_0)$ for at least one $n \geq 0$ with $\gamma_n(a_0) \neq 0$}\}.
$$
Recall that the natural density of a set $T$ of primes in $\O_K$ is
\begin{equation*}\label{densedef}
D(T) = \lim_{x \to \infty} \frac{\#\{\p \in T : N(\p) \leq x\}}{\#\{\p : N(\p) \leq x\}},
\end{equation*}
provided that this limit exists. Here $N(\p)$ denotes the norm of $\p$. Likewise, we define the upper density $D^+(T)$ by replacing the limit above with a $\limsup$.      
In particular, our goal is to determine $D(P(\gamma, a_0))$ in certain cases. 

As a first step, we relate $D(P(\gamma, a_0))$ to the Galois group $G_{\gamma, n, K}$ of $K_n(\gamma)/K$. Throughout this section, we suppress the dependence on $\gamma$ and $K$ and write $G_n$ in place of $G_{\gamma, n, K}$. The group $G_n$ acts naturally on the roots $R_n$ of $\gamma_n$ (note that $R_n$ is the $n$th level of the tree $T_{\gamma,n},$ which we denote just by $T_n$). Define the \textit{fixed-point proportion} of $G_n$ to be
\begin{equation} \label{fppdef}
\FPP(G_n) := \frac{\#\{g \in G_n : \text{$g$ fixes at least one element of $R_n$}\}}{\#G_n}.
\end{equation}
Note that since $G_n$ acts on $T_n$ by tree automorphisms for every $n$, we have that the sequence $(\FPP(G_n))_{n \geq 1}$ is non-increasing, and hence its limit must exist. 

The following theorem is a version of \cite[Theorem 2.1]{Jones}, adapted to the present circumstances. 
\begin{theorem} \label{density to Galois}
Assume that $\gamma_n$ is separable over $K$ for all $n \geq 1$, and let $a_0 \in K$. Then 
\begin{equation} \label{ineq}
D^+(P(\gamma, a_0)) \leq \lim_{n \to \infty} \textup{FPP}(G_n).
\end{equation}
\end{theorem}


\begin{remark}
Our applications of Theorem \ref{density to Galois} are all in the case where $\textup{FPP}(G_n) \to 0$, in which case \eqref{ineq} implies that $D(P(\gamma, a_0))$ exists (and equals 0). In general when 
$\lim_{n \to \infty} \textup{FPP}(G_n) > 0$ we do not a priori know that $D(P(\gamma, a_0))$ exists. In the setting of iteration of a single rational map $\phi$, there are known cases where $\lim_{n \to \infty} \textup{FPP}(G_n) > 0$, and they occur when $\phi$ is a finite quotient of an affine map of an abelian algebraic group. In this situation the extra structure often allows one to show directly that the relevant density of prime divisors exists (and equals $\lim_{n \to \infty} \textup{FPP}(G_n)$). See e.g. \cite[Theorem 3.2]{jones-rouse}.
\end{remark}

\begin{proof} Denote the discriminant of $\gamma_n$ by $\text{Disc}(\gamma_n)$, and for a prime $\p$ of $\O_K$ let $v_\p$ be the $\p$-adic valuation. Let $B_n$ be the following finite set of primes of $\O_K$:
$$
B_n = \{\p : \text{$\p \mid \text{Disc}(\gamma_n)$ or $v_\p(c) < 0$ for some coefficient $c$ of $\gamma_n$}\},
$$
and note that if $\p \not\in B_n$ then $\gamma_n$ has good reduction, i.e. reducing $\gamma_n$ coefficient-wise modulo $\p$ yields a polynomial $\widetilde{\gamma_n}$ in $\O_K/\p$ of degree $\deg(\gamma_n)$. In particular, every root $\alpha$ of $\gamma_n$ satisfies $v_\p(\alpha) \geq 0$. Fix $n \geq 1$, and let
\begin{align*}
\Omega_n & = \{\p: \text{$\p \not\in B_n$ and $\gamma_n(x) \equiv 0 \bmod{\p}$ has no solution in $K$}\}, \\
R_n & = \{\text{$\p : \p \not\in B_n$ and $\p \nmid \gamma_N(a_0)$ for all $N \geq n$}\}.
\end{align*}

Assume that $\p \in \Omega_n$ and take $N > n$. If $\gamma_N(x) \equiv 0 \bmod{\p}$ has a solution in $K$, then because $\gamma_N = \gamma_n \circ \theta_{n+1} \circ \cdots \circ \theta_N$, it follows that $\gamma_n(x) \equiv 0 \bmod{\p}$ has a solution in $K$, and thus $\Omega_n \subseteq R_n$. 
Observe that there are only finitely many prime $\p$ with 
$\p \mid \gamma_N(a_0)$ and $\gamma_N(a_0) \neq 0$  for some $N < n$, and this together with the finiteness of $B_n$ imply that $R_n$ and the complement $P(\gamma, a_0)^c$ differ by a finite set of primes, which we denote by $F_n$. 

Assume for a moment that $D(\Omega_n)$ exists and equals $d_n$. Let $\epsilon > 0$ and take $x_0$ large enough so that if $x \geq x_0$ then both 
\begin{equation} \label{specifics}
\frac{\#\{\p \in \Omega_n : N(\p) \leq x\}}{\#\{\p : N(\p) \leq x\}} > d_n - \epsilon/2 \quad \text{and}  \quad \frac{\#\{\p \in F_n : N(\p) \leq x\}}{\#\{\p : N(\p) \leq x\}} < \epsilon/2. 
\end{equation}
The first inequality in \eqref{specifics} implies that $\frac{\#\{\p \in R_n : N(\p) \leq x\}}{\#\{\p : N(\p) \leq x\}} > d_n - \epsilon/2$, and together with the second inequality in \eqref{specifics} this gives 
$$
\frac{\#\{\p \in P(\gamma, a_0)^c : N(\p) \leq x\}}{\#\{\p : N(\p) \leq x\}} > d_n - \epsilon.
$$
This implies that 
$
\frac{\#\{\p \in P(\gamma, a_0) : N(\p) \leq x\}}{\#\{\p : N(\p) \leq x\}} < (1 - d_n) + \epsilon,
$
from which we obtain $D^+(P(\gamma, a_0)) < 1-d_n$, and thus $D^+(P(\gamma, a_0)) \leq \lim_{n \to \infty} (1-d_n)$.

The proof will be complete once we show that $D(\Omega_n)$ exists and equals $1 - \FPP(G_n)$. 
If $\p \nmid \text{Disc}(\gamma_n)$, then $\p$ cannot divide the field discriminant of $K_n(\gamma)/K$; see, for instance, \cite[Corollary 2, p. 157]{narkiewicz}. Hence, such $\p$ are unramified in the extension $K_n(\gamma)/K$.  Now
$\gamma_n(x) \equiv 0 \pmod{\p}$ having a solution in $K$ is equivalent to $\widetilde{\gamma_n}$ having at least one linear factor in $(\O_K/\p)[x]$.  Except for possibly finitely many $\p$, this implies that 
$\p\O_L = \mathfrak{P}_1 \cdots \mathfrak{P}_r$, where $L/K$ is obtained by adjoining a root of $\gamma_n$, $\O_L$ is the ring of integers of $L$, and at least one of the $\mathfrak{P}_i$ has residue class degree one \cite[Theorem 4.12]{narkiewicz}. This is equivalent to the disjoint cycle decomposition of the Frobenius conjugacy class at $\p$ having a fixed point (in the natural permutation representation of $G_n$ on the roots $R_n$ of $\gamma_n$). From the Chebotarev Density Theorem it follows \cite[Proposition 7.15]{narkiewicz} that the density of $\p$ with $\p\O_L$ having such a decomposition exists and equals $\FPP(G_n)$. We have thus shown that $D(\Omega_n^c)$ exists and equals $\FPP(G_n)$, and the proof is complete. 
\end{proof}

We now tackle the problem of computing $\lim_{n \to \infty} \FPP(G_n)$ for certain choices of $\gamma_n$. A convenient vehicle for this is a stochastic process that encodes fixed-point information about the action of elements of $G_\infty := \varprojlim G_n$ on $R_n$ for each $n$. Let $\textbf{P}$ be the Harr measure on $G_\infty$, normalized so that $\textbf{P}(G_\infty) = 1$. Each $g \in G_\infty$ acts on $R_n$ for all $n \geq 1$. We define random variables $X_1, X_2, \ldots$ by 
$$
X_n(g) = \text{number of elements of $R_n$ fixed by $g$}.
$$
We call the stochastic process $(X_1, X_2, \ldots)$ the \textit{fixed-point process} of $G_\infty$. In particular, note that $\FPP(G_n) = \textbf{P}(X_n > 0)$.

\begin{definition} \label{martdef}
A stochastic process $(X_0, X_1, X_2, \ldots)$ taking values in $\mathbb{Z}$ is a {\em martingale} if for all $n \geq 1$ and any $t_i \in \mathbb{Z}$, 
\[E(X_n \mid X_{0} = t_{0}, X_1 = t_1, \ldots, X_{n-1} = t_{n-1}) = t_{n-1}.\]
We call $(X_0, X_1, X_2, \ldots)$ an {\em eventual martingale} if for some $n_0 \geq 1$ the tail end process
$X_{n_0}, X_{n_0 + 1}, X_{n_0 + 2}, \ldots $ is a martingale.  
\end{definition}
A standard martingale convergence theorem shows the following agreeable property of eventual martingales, which is Corollary 2.3 of \cite{Jones}:
\begin{proposition} \label{evconstprop}
Suppose that the Galois process of $G_\infty$ is an eventual martingale.  Then
\[\textbf{P}(\{g \in G_\infty : \text{$X_1(g), X_2(g), \ldots$ is eventually constant}\}) = 1.\]
\end{proposition}

Denote the Galois group of $K_n(\gamma)/K_{n-1}(\gamma)$ by $H_n$. Proposition \ref{evconstprop} allows us to obtain significant information about the Galois process of $G_\infty$ from knowledge of only an infinite number of $H_n$. The following is Lemma 5.3 of \cite{Jones-io}:
\begin{lemma}[\cite{Jones-io}] \label{maxlem}
Assume that all $\theta_i$ are quadratic, and let $n \geq 2$. If $H_n$ is maximal, then for any $m$ with $1 \leq m < n$ and any integer $u > 0$, 
\begin{equation} \label{cond}
\textbf{P}(X_n = u \mid X_m = u, X_{m+1} = u, \ldots, X_{n-1} = u) \leq \frac{1}{2}.
\end{equation}
\end{lemma}
Indeed, if $u$ is not of the form $2w$ for $1 \leq w \leq 2^{m-1}$, then $\textbf{P}(X_n = u) = 0$. Otherwise, Lemma 5.2 of \cite{Jones-io} shows that the maximality of $H_n$ forces $(X_n \, | \, (X_{n-1} = u))$ to have the same distribution as flipping $u$ fair coins, with heads counting 2 and tails 0. Thus the left-hand side of \eqref{cond} is 
${u\choose{u/2}} \frac{1}{2^u}$, which is easily seen to be at most $1/2$.

To make use of Proposition \ref{evconstprop} and Lemma \ref{maxlem}, we wish to give conditions on $\gamma$ under which the fixed point process of $G_\infty$ is an eventual martingale. We make the assumption that all $\theta_i$ are quadratic, and use arguments similar to those in \cite[Section 2]{Jones}. 

\begin{theorem} \label{evmartthm}
Suppose that for all $n \geq 1$, $\theta_n$ is quadratic and $\gamma_n$ is irreducible over $K$. If $H_n$ is non-trivial for all $n$ sufficiently large, then the fixed point process of $G_\infty$ is an eventual martingale. 
\end{theorem}

\begin{proof}
First note that since $K$ is a number field, the irreducibility of $\gamma_n$ over $K$ implies that $\gamma_n$ is separable over $K$. From Lemma 2.4 and Theorem 2.5 of \cite{Jones} it follows that the present theorem is proved provided that the following holds: there exists $n_0 \geq 1$ such that for all $n \geq n_0$ and every root $\alpha$ of $\gamma_{n-1}$, the polynomial $\theta_n - \alpha$ is irreducible over $K_{n-1}(\gamma)$. Let $n_0$ be such that $H_n$ is non-trivial for $n \geq n_0$, and fix $n \geq n_0$. Observe that 
\begin{equation} \label{rootspart}
\gamma_n = \prod_{\alpha \in \gamma_{n-1}^{-1}(0)} (\theta_n(x) - \alpha).
\end{equation}
Because $H_n$ is non-trivial, there must be some $\alpha$ such that $\theta_n - \alpha$ has a root not in $K_{n-1}(\gamma)$. Because $\theta_n$ has degree 2, this implies that $\theta_n - \alpha$ is irreducible over $K_{n-1}(\gamma)$. By hypothesis, $\gamma_{n-1}$ is irreducible over $K$, and hence $G_{n-1}$ acts transitively on the roots of $\gamma_{n-1}$. It follows that $\theta_n - \alpha$ is irreducible over $K_{n-1}(\gamma)$ for all $\alpha$, as desired. 
\end{proof}

We are now in a position to prove the main result of this section, which we restate from the Introduction 

\thmdensity*
\begin{proof}
As in the discussion above, write $G_\infty$ for $G_{\gamma, K}$, the inverse limit of the $G_n$, and write $H_n$ for the Galois group of $K_n(\gamma)/K_{n-1}(\gamma)$. We first claim that $H_n$ is non-trivial for all sufficiently large $n$, which by Theorem \ref{evmartthm} shows that the fixed point process of $G_\infty$ is an eventual martingale. 

Let $\C$ be the partition of $T_n$ into the sets $\{\text{roots of $\theta_n(x) - \alpha$}\},$ where $\alpha$ varies over roots of $\theta_n$, as in \eqref{rootspart}.  Let $\sigma_{\C} \in {\rm Sym} (V_n)$ be the permutation associated to 
$\C$, i.e. the unique permutation whose orbits are precisely the sets belonging to $\C$.  We wish to show that $\sigma_{\C} \in G_n$ for $n$ sufficiently large, and then the fact that $\sigma_\C$ only acts non-trivially on fibers of $\theta_n$ implies that $\sigma_\C \in H_n$, whence $H_n$ is non-trivial.  

Because $S$ contains only quadratic polynomials, $G_n$ is a 2-group, and thus has non-trivial center.  Let $\delta$ be a nontrivial element of the center of $G_n$, and $\D$ the corresponding central fiber system, i.e. the collection of orbits of $\delta$ (see \cite[Proposition-Definition 4.10]{Jones-io}).  Then by \cite[Theorem 4.9]{Jones-io} we have 
either $\sigma_{\C} = \delta$ or $G_n$ is composed entirely of alternating permutations.  

In the latter case, $\text{Disc}(\gamma_n)$ is a square in $K$. Now for $n \geq 2$, the discriminant formula in \cite[Proposition 6.2]{Me:LeftRightTotal} and the fact that all elements of $S$ are monic and of even degree imply that $\pm \gamma_n(0)$ must be a square in $K$. This implies that the either the curve $C^+ : y^2 = \gamma_2(x)$ or $C^{-} : y^2 = -\gamma_2(x)$ has a point $P$ with $x(P) = 0$ for $n = 2$ and $x(P) = (\theta_3 \circ \cdots \circ \theta_n)(0)$ for $n \geq 3$. The coordinates of $P$ must have non-negative $\p$-adic valuation for every prime $\p$ of $\O_K$ except for the finitely many where $v_\p(c_i) < 0$ for some $i$. 

By hypothesis the degree-4 polynomial $\gamma_2(x)$ is separable, so me may apply Siegel's Theorem \cite[p.353]{jhsdioph} to conclude that both $C^+$ and $C^-$ have only finitely many such points $P$, and hence $\gamma_n(0)$ is a square in $K$ for only finitely many $n$. We have thus shown that $\sigma_{\C} = \delta$ for $n$ sufficiently large, and because $\delta \in G_n$, this implies that $\sigma_{\C} \in G_n$. Hence $H_n$ is non-trivial for $n$ sufficiently large, and thus the fixed point process of $G_\infty$ is an eventual martingale. 

We may now apply both Proposition \ref{evconstprop} and Lemma \ref{maxlem}. An argument identical to the proof of Theorem 1.3 on p. 1122 of \cite{Jones-io} shows that
$$
\lim_{n \to \infty} \textbf{P}(X_n > 0) = 0
$$
By the definition of the Galois process, this is the same as $\lim_{n \to \infty} \FPP(G_n) = 0$, and so Theorem \ref{density to Galois} gives $D^+(P(\gamma, a_0)) = 0$, whence $D(P(\gamma, a_0)) = 0$.
\end{proof}
\begin{remark} It is worth pointing out that the density results in this section rely on some probability theory (i.e., martingales), but in a totally different way from the probability theory used in Sections \ref{sec:finiteorbit}-\ref{sec:Z[t]} (e.g., the monkey and typewriter problem). Thus, we have provided at least two distinct applications of probability theory to the study of arboreal representations. 
\end{remark}
  
  \bibliography{bibfile}
  \bibliographystyle{alpha}

\end{document}